\documentclass[reqno]{amsart}
\newtheorem{theorem}{Theorem}[section]

\newtheorem{lemma}[theorem]{Lemma}
\newtheorem{proposition}[theorem]{Proposition}
\theoremstyle{definition}
\usepackage{fullpage}
\usepackage[dvips]{graphicx}
\usepackage{accents}
\usepackage{array}
\usepackage{subcaption}
\usepackage{float}
\newcommand{\ve}{\varepsilon}
\newcommand{\abs}[1]{\lvert#1\rvert}
\newcommand{\acr}{\accentset{\circ}}
\newcommand{\wt}[1]{\widetilde{#1}}

\numberwithin{equation}{section}
\numberwithin{figure}{section}
\def\ontop#1#2{\setbox0\hbox{#2}\copy0\llap{\raise\ht0\hbox{#1}}}
\begin{document}
\newcommand{\shw}{\ensuremath{\overset{\scriptsize \circ}{S}}}
\title[]{ON THE STANDARD GALERKIN
METHOD WITH EXPLICIT RK4 TIME STEPPING FOR THE SHALLOW WATER EQUATIONS.}
\author{D.C. Antonopoulos}
\address{Department of Mathematics, University of Athens, 15784 Zographou, Greece}
\email{antonod@math.uoa.gr}
\author{V.A. Dougalis}
\address{Department of Mathematics, University of Athens, 15784 Zographou, Greece, and
Institute of Applied and Computational Mathematics, FORTH, 70013 Heraklion, Greece}
\email{doug@math.uoa.gr}
\author{G. Kounadis}
\address{Department of Mathematics, University of Athens, 15784 Zographou, Greece, and
Institute of Applied and Computational Mathematics, FORTH, 70013 Heraklion, Greece}
\email{gregk@math.uoa.gr}
\subjclass[2010]{65M60, 65M12;}
\keywords{Shallow water equations, standard Galerkin finite element method,
$4^{th}$ order, $4$-stage explicit Runge-Kutta method, error estimates.}
\begin{abstract}
We consider a simple initial-boundary-value problem for the shallow water equations in one
space dimension. We discretize the problem in space by the standard Galerkin finite element 
method on a quasiuniform mesh and in time by the classical $4$-stage, $4^{th}$ order,
explicit Runge-Kutta scheme. Assuming smoothness of solutions, a Courant number restriction,
and certain hypotheses on the finite element spaces, we prove $L^{2}$ error estimates that
are of fourth-order accuracy in the temporal variable and of the usual, due to the nonuniform
mesh, suboptimal order in space. We also make a computational study of the numerical spatial
and temporal orders of convergence, and of the validity of a hypothesis made on the finite
element spaces.    
\end{abstract}
\maketitle
\section{Introduction}
In this paper we will consider the following initial-boundary-value problem (ibvp) for the
{\emph{shallow water equations} posed on the spatial interval $[0,1]$. For $T>0$ we seek
$\eta=\eta(x,t)$, $u=u(x,t)$, $0\leq x\leq 1$, $0\leq t\leq T$, such that  
\begin{align}
\begin{aligned}
\eta_{t} & + u_{x} + (\eta u)_{x} = 0, \notag\\
u_{t} & + \eta_{x} + uu_{x} = 0,\notag
\end{aligned}
\quad  0\leq x\leq 1,\,\,\, 0\leq t\leq T,
\tag{SW} \label{eqsw}
\\
\eta(x,0) =\eta_{0}(x), \quad u(x,0)=u_{0}(x), \quad 0\leq x\leq 1,\hspace{-9pt}&  \notag\\
u(0,t) = 0,\quad u(1,t)=0, \quad 0\leq t\leq T, \hspace{32pt} & \notag
\end{align}
where $\eta_{0}$, $u_{0}$ are given real-valued functions defined on $[0,1]$. The shallow
water equations approximate the Euler equations of water wave theory in the case of long
waves in a channel of finite depth. In \eqref{eqsw} the variables are nondimensional
and unscaled; $x\in [0,1]$ and $t\geq 0$ are proportional to position along the finite 
channel $[0,1]$ and time, respectively, $\eta = \eta(x,t)$ is proportional to the elevation
of the free surface above a level of rest corresponding to $\eta = 0$, and $u=u(x,t)$ is 
proportional to the depth-averaged horizontal velocity of the fluid. In these variables the
(horizontal) bottom of the channel is at a depth equal to $-1$. \par
Even if the initial conditions $\eta_{0}$ and $u_{0}$ are smooth, \eqref{eqsw} is not expected
to have global smooth solutions. There is however a local $H^{2}$ well-posedness theory; in     
\cite{pt} Petcu and Temam proved that if $u_{0}\in H^{2}\cap\acr{H}^{1}$, 
$\eta_{0}\in H^{2}$, with $1 + \eta_{0}\geq 2\alpha>0$ on $[0,1]$ for some constant
$\alpha$, then there exists a $T = T(\|u_{0}\|_{2}, \|\eta_{0}\|_{2})>0$ and a unique
solution $(\eta, u) \in L^{\infty}(0,T;H^{2}\times(H^{2}\cap\acr{H}^{1}))$ of \eqref{eqsw}
such that  $1 + \eta \geq \alpha > 0$ on $[0,1]\times[0,T]$. (Here, and in the sequel, for
integer $m\geq 0$ $H^{m}$, $\|\cdot\|_{m}$, denote the $L^{2}$-based Sobolev space of 
functions on $[0,1]$ and its associated norm, and $\acr{H}^{1}$ the subspace of $H^{1}$
whose elements are equal to zero at $x=0$, $1$. For a Banach space $X$ of functions
on $[0,1]$, $L^{\infty}(0,T;X)$ is the space of $L^{\infty}$ maps from $[0,T]$ to $X$.
\par
In the paper at hand we will approximate the solution of \eqref{eqsw} by a fully discrete scheme using the standard Galerkin finite element method for the discretization in space
with suitable finite element spaces, whose elements are at least continuously differentiable
on $[0,1]$ and are piecewise polynomial functions of degree $r-1$, $r\geq 3$, with respect
to a quasiuniform partition of $[0,1]$ of maximum meshlength $h$. Precise assumptions about the
finite element spaces will be stated in section 2. For the temporal integration we will use
the classical, four-stage, fourth-order, explicit Runge-Kutta scheme with a 
uniform time step $k$. In section 3 we analyze the spatial and temporal consistency and in section 4 the convergence of the scheme. Specifically we show that if the solution of 
\eqref{eqsw} is sufficiently smooth and $1 + \eta$ is positive in $[0,1]\times[0,T]$, there
exists a positive constant $\lambda_{0}$ such that if the Courant number $\lambda = k/h$ satisfies $\lambda \leq \lambda_{0}$, then the $L^{2}$ norm of the error of the fully
discrete approximation is of $O(k^{4} + h^{r-1})$. (It is well known that the best order of spatial accuracy one may achieve for first-order hyperbolic problems using the
standard Galerkin method on a nonuniform mesh is $r-1$ in general.). In section 5 we make
a computational study of the numerical spatial and temporal orders of convergence and of
the validity of a certain hypothesis made on the finite element spaces. 
\par
Explicit Runge-Kutta (RK) methods of higher (at least third) order of accuracy have
been widely used for the temporal discretization of ode systems obtained from spatial
discretizations of first-order hyperbolic equations. Such ode systems are usually only
mildly stiff and may be stably integrated with explicit RK schemes under Courant-number restrictions.
Regarding rigorous error estimates for fully discrete schemes of finite element-high order 
RK type we mention the paper \cite{zs1} by Zhang and Shu, who prove error estimates for
a fully discrete DG - 3$^{\mathrm{d}}$ order Shu-Osher RK scheme, cf. \cite{so}, for scalar 
conservation laws. The same authors analyze in \cite{zs2} a similar fully discrete 
scheme applied to a scalar linear hyperbolic equation with discontinuous initial condition.
In \cite{bef} Burman \emph{et al.} consider ibvp's for first-order linear hyperbolic 
problems of Friedrichs type in several space dimensions, discretized in space by a class
of symmetrically stabilized finite element methods that includes DG schemes, and in time 
by, among other, third-order accurate, explicit RK schemes, and prove $L^{2}$-error
estimates of optimal order in time and quasioptimal $(r-1/2)$ in space. In \cite{ad} two of the present authors 
proved, among other, $O(k^{3} + h^{r-1})$ $L^{2}$-error estimates for \eqref{eqsw} 
discretized by the standard Galerkin method coupled with the Shu-Osher RK scheme. 
For practical issues regarding the application of DG-high order RK schemes to nonlinear
hyperbolic systems including the shallow water equations, we refer the reader to the 
recent review papers \cite{qz} and \cite{x}, and to \cite{kyk} and its references on the
strong stability of higher order RK schemes.
\par
In addition to previously introduced notation, in the sequel we let $C^{m}=C^{m}[0,1]$,
$m=0,1,2,...$, be the space of $m$ times continuously differentiable functions on $[0,1]$.
The inner product and norm on $L^{2}=L^{2}(0,1)$ will be denoted by $(\cdot,\cdot)$,
$\|\cdot\|$, respectively, while the norms of $L^{\infty}=L^{\infty}(0,1)$ and of the
$L^{\infty}$-based Sobolev space $W^{1,\infty} = W^{1,\infty}(0,1)$ by 
$\|\cdot\|_{\infty}$, $\|\cdot\|_{1,\infty}$. We let $\mathbb{P}_{r}$ be the polynomials
of degree at most $r$.
\section{Approximation properties of the finite element spaces and preliminaries}
Let $0=x_{1}<x_{2}<\dots<x_{N+1}=1$ be a quasiuniform partition of $[0,1]$ with
$h:=\max_{i}(x_{i+1}-x_{i})$. For integers $r$, $\mu$ with $r\geq 3$ and 
$1\leq \mu\leq r-2$, let 
$S_{h}=S_{h}^{r,\mu}=\{\phi \in C^{\mu} : 
\phi\big|_{[x_{i},x_{i+1}]} \in \mathbb{P}_{r-1}\}$\,,
and $S_{h,0} = \{\phi \in S_{h},\, \phi(0) = \phi(1) = 0\}$. We will assume, cf. \cite{debf},
\cite{schr}, that if $w\in H^{s}$, $2\leq s\leq r$, there exists a $\chi \in S_{h}$,
such that
\begin{subequations}
\[
\|w-\chi\| + h\|w' - \chi'\|\ \leq Ch^{s}\|w^{(s)}\|,
\tag{2.1a}
\label{eq21a}
\]
%\end{subequations}
and that if  $w \in H^{s}$ , $3\leq s\leq r$, $\chi$ satisfies in addition
\[
\|w - \chi\|_{2}  \leq C h^{s-2}\|w^{(s)}\|,
\label{eq21b}
\tag{2.1b}
\]
\end{subequations}
for some constant $C$ independent of $h$ and $w$. We will also assume that similar properties hold for $S_{h,0}$ if $w$ satisfies in addition $w(0) = w(1) = 0$. Well-known examples of spaces satisfying $(2.1a-b)$ include the Hermite piecewise polynomial
functions, for which $r=2\mu + 2$, \cite{bsw}, and the spaces of smooth  splines of even order
(i.e. piecewise polynomial of odd degree), for which $r = \mu + 2$, where $\mu \geq 2$ is even,
\cite{s}. (For smooth splines of any order $r = \mu + 2$, $\mu \geq 1$, \eqref{eq21b} holds at
least for uniform meshes, cf. e.g. \cite{bs} and its references.)
\par
Note that, as a consequence  of the quasiuniformity of the mesh, the following inverse
inequalities hold for $\chi \in S_{h}$ or $\chi \in S_{h,0}$
\begin{equation}
\begin{aligned}
\|\chi\|_{\alpha} & \leq C h^{-(\alpha - \beta)}\|\chi\|_{\beta},
\quad 0\leq \beta\leq \alpha \leq \mu + 1, \\
\|\chi\|_{j,\infty} & \leq Ch^{-(j + 1/2)}\|\chi\|, \quad 0\leq j\leq \mu,
\end{aligned}
\label{eq22}
\end{equation}
for constants $C$ independent of $h$ and $\chi$. Also, as a consequence of (2.1a-b) and the
quasiuniformity of the mesh, it follows that if $P$ is the $L^{2}$-projection operator onto
$S_{h}$, then the following hold, \cite{tw}, \cite{ddw},
\begin{subequations}
\begin{align*}
& \|Pv\|_{1} \leq C \|v\|_{1}, \quad \forall v \in H^{1},
\tag{2.3a}
\label{eq23a} \\
& \|Pv\|_{\infty} \leq C \|v\|_{\infty}, \quad \forall v \in C^{0},
\tag{2.3b}
\label{eq23b} \\
& \|Pv - v\|_{\infty} \leq C h^{r}\|v\|_{r,\infty}, \quad \forall v \in C^{r},
\tag{2.3c}
\label{eq23c}
\end{align*}
\end{subequations}
for some constants $C$ independent of $h$ and $v$. The same inequalities hold for the
$L^{2}$-projection operator $P_{0}$ onto $S_{h,0}$ when, in addition, $v(0) = v(1) = 0$.
(In the sequel we shall refer to the analogous results for $P_{0}$ on $S_{h,0}$ using the same
formula numbers, i.e. (2.3a-c).) \par
The standard Galerkin method for the semidiscretization of (SW) is defined as follows:
We seek $\eta_{h} : [0,T]\to S_{h}$, $u_{h} : [0,T]\to S_{h,0}$, such that for $t\in[0,T]$
\begin{equation}
\begin{aligned}
(\eta_{ht},\phi) & + (u_{hx},\phi) + \bigl((\eta_{h}u_{h})_{x}, \phi) = 0,
\quad \forall \phi \in S_{h},\\
(u_{h},\chi) & + (\eta_{hx},\chi) + (u_{h}u_{hx},\chi) = 0,
\quad \forall \chi \in S_{h,0},
\end{aligned}
\label{eq24}
\end{equation}
with initial conditions
\begin{equation}
\eta_{h}(0) = P \eta_{0}, \quad u_{h}(0) = P_{0}u_{0}.
\label{eq25}
\end{equation}
In \cite[Proposition 2.1]{ad} it was proved that if $(\eta,u)$, the solution of (SW), is
sufficiently smooth and satisfies $1+\eta > 0$ for $t\in [0,T]$, and if $r\geq 3$ and $h$ is
sufficiently small, then the semidiscrete ivp \eqref{eq24}-\eqref{eq25} has a unique solution
$(\eta_{h},u_{h})$ for $t\in [0,T]$ satisfying
\begin{equation}
\max_{0\leq t\leq T}(\|\eta(t) - \eta_{h}(t)\| + \|u(t) - u_{h}(t)\|) \leq C h^{r-1}.
\label{eq26}
\end{equation}
It is well known that $r-1$ is the best order of convergence in $L^{2}$ expected for the standard
Galerkin method for first-order hyperbolic problems on general quasiuniform meshes; for a uniform
mesh better rates of convergence may be obtained, \cite{d1}. For uniform meshes it was proved in
\cite{ad} that in the case of (SW) one obtains $O(h^{2})$ $L^{2}$-convergence for the semidiscrete
approximation with continuous, piecewise linear functions. In the case of the periodic ivp for the
shallow water equations the semidiscrete approximation with smooth splines on a uniform mesh gives
an optimal-order $L^{2}$ error estimate of $O(h^{r})$, cf. \cite{ad}. The assumption that $r\geq 3$
is needed in the proof of \eqref{eq26} in order to control the $W^{1,\infty}$ norm of an error term,
and was also present in the error analysis of \cite{d2} for a close relative of the SW system.
(Numerical experiments in \cite{ad} on quasiuniform meshes  suggest that \eqref{eq26} holds for
continuous, piecewise linear functions ($r=2$) as well; hence the assumption $r\geq 3$ may be
technical.) \par
In the analysis of the fully discrete scheme under consideration we will assume that
$r\geq 3$ and that the mesh is quasiuniform, so that the spatial error in $L^{2}$ will be
$O(h^{r-1})$. The emphasis of the convergence proof will be placed in getting the optimal temporal-order $L^{2}$ error estimate $O(k^{4} + h^{r-1})$. In the proof, 
the fully discrete approximations
will not be compared to the semidiscrete solution but directly to the $L^{2}$ projection of
the solution of the continuous problem (SW). Thus, the semidiscretization will not be further
utilized in this paper. \noindent
In the sequel we will assume that (SW) possesses a unique, sufficiently smooth solution $(\eta,u)$
for $0\leq t\leq T$, such that $1 + \eta \geq \alpha >0$ for $(x,t)\in [0,1]\times [0,T]$ for some
constant $\alpha$. We will denote by $C$ positive constants independent of the discretization
parameters. \par
In the proofs of sections 3 and 4 we will make use of several estimates that follow
from the assumptions on the approximation and inverse properties of the finite element spaces made
thusfar. One of them is the following {\em{superapproximation property}} of $S_{h,0}$, \cite{ddw},
\cite{d2}, according to which
\begin{equation}
\|P_{0}[(1+\eta)\xi] - (1+\eta)\xi\| \leq Ch \|\xi\|, \quad \forall \xi \in S_{h,0}.
\label{eq27}
\end{equation}
We will also use the following results, that we state as Lemmata.
\begin{lemma} Let $H=P\eta$. \\
(i)\, Then
\begin{equation}
\|P_{0}[(1+H)\xi] - (1+H)\xi\| \leq Ch\|\xi\|, \quad \forall \xi \in S_{h,0}.
\label{eq28}
\end{equation}
(ii)\, If $f\in L^{2}(0,1)$ and
\begin{equation}
((1+H)\xi,P_{0}f) = ((1+H)\xi,f) + b(\xi,f), \quad \xi \in S_{h,0},
\label{eq29}
\end{equation}
then $\abs{b(\xi,f)} \leq Ch\|\xi\| \|f\|$.
\label{L21}
\end{lemma}
\begin{proof} (i)\, We have
\begin{align*}
P_{0}[(1+H)\xi] - (1+H)\xi & = P_{0}[(1+H)\xi] - P_{0}[(1+\eta)\xi]
+ P_{0}[(1+\eta)\xi] - (1+\eta)\xi \\
& \,\,\,\,\,\, + (1+\eta)\xi - (1+H)\xi \\
& = P_{0}[(H - \eta)\xi] + P_{0}[(1+\eta)\xi] - (1+\eta)\xi - (H-\eta)\xi,
\end{align*}
whence, from \eqref{eq23b}, \eqref{eq29},
\[
\|P_{0}[(1+H)\xi] - (1+H)\xi\| \leq C(\|H - \eta\|_{\infty}\|\xi\| + h\|\xi\|),
\]
and therefore \eqref{eq28} follows from \eqref{eq23c}. \\
(ii)\, We have
\[
b(\xi,f) = (P_{0}[(1+H)\xi],f) - ((1+H)\xi,f) = (P_{0}[(1+H)\xi] - (1+H)\xi,f),
\]
and therefore, by \eqref{eq28}, $\abs{b(\xi,f)} \leq Ch\|\xi\| \|f\|$.
\end{proof}
\begin{lemma} Let $\eta$ be the first component of the solution of
\eqref{eqsw} for which we suppose that
$1+ \eta \geq \alpha > 0 $, and $H = P\eta$. If $\eta \in C^{r}$, then for sufficiently small $h$
we have
\[
1 + H \geq \tfrac{\alpha}{2}.
\]
In addition, if $f\in L^{2}(0,1)$ then
\begin{equation}
\tfrac{\alpha}{2}\|f\|^{2} \leq ((1 + H)f,f) \leq C'\|f\|^{2},
\label{eq210}
\end{equation}
for some constant $C'$ depending on $\|\eta\|_{r,\infty}$.
\label{L22}
\end{lemma}
\begin{proof} From \eqref{eq23c} we have
\[
1 + \eta - C_{1}h^{r} \leq 1 + H \leq 1 + \eta + C_{1}h^{r},
\]
for some constant $C_{1}$. Therefore if $h\leq (\alpha/(2C_{1}))^{1/r}$ then
$\alpha/2 \leq 1 + H \leq C^{'}$, from which \eqref{eq210} follows.
\end{proof}
For the purposes of the proof of convergence of the fully discrete scheme we will also need two more
properties of the $L^{2}$-projection operators $P$ and $P_{0}$, in addition to (2.3a-c). The first
one follows from the approximation and inverse properties of the finite element spaces already
mentioned. It expresses the fact that $P$ is {\em{stable in $H^{2}$}}, i.e. that there exists a
constant $C$ such that
\begin{equation}
\|Pv\|_{2} \leq C \|v\|_{2}, \quad \forall v \in H^{2}.
\label{eq211}
\end{equation}
In addition, the analogous stability estimate holds for $P_{0}$ on $v \in H^{2}\cap\acr{H}^{1}$.
It is straightforward to check that \eqref{eq211} follows from the hypotheses on $S_{h}$ made thusfar. Indeed, let $R_{h} : H^{1} \to S_{h}$ be the $H^{1}$-projection 
onto $S_{h}$ defined for
$w\in H^{1}$ by $(R_{h}w,\phi)_{1} = (w,\phi)_{1}$ for all $\phi\in S_{h}$ Suppose $v\in H^{2}$
and let $\psi$ be the interpolant of $v$ in the space of continuous, piecewise linear functions
defined with respect to the partition $\{x_{i}\}_{i=1}^{N+1}$. Then, by a local inverse inequality for $R_{h}v - \psi \in \mathbb{P}_{r-1}(x_{j},x_{j+1})$ and the 
quasiuniformity of the mesh we have
\[
\|(R_{h}v)''\|^{2} = \sum_{j=1}^{N}\int_{x_{j}}^{x_{j+1}}((R_{h}v)'' - \psi'')^{2}
\leq Ch^{-2}\sum_{j=1}^{N}\int_{x_{j}}^{x_{j+1}}((R_{h}v)' - \psi')^{2}.
\]
Hence $\|(R_{h}v)''\| \leq Ch^{-1}\|(R_{h}v)' - \psi'\|
\leq Ch^{-1}(\|R_{h}v - v\|_{1} + \|v - \psi\|_{1})$. Since
$\|R_{h}v - v\|_{1} \leq C\inf_{\chi\in S_{h}}\|v - \chi\|_{1}\leq Ch\|v\|_{2}$,
by \eqref{eq21a}, it follows that $\|(R_{h}v)''\| \leq C\|v\|_{2}$, from which the stability
of $R_{h}$ in $H^{2}$ follows in view of the fact that $\|R_{h}v\|_{1} \leq \|v\|_{1}$, $v\in H^{1}$.
Finally,
\begin{align*}
\|Pv\|_{2} & \leq \|Pv - R_{h}v\|_{2} + \|R_{h}v\|_{2} \leq Ch^{-2}\|P(v - R_{h}v)\|
+ C\|v\|_{2} \\
& \leq Ch^{-2}\|v - R_{h}v\| + C\|v\|_{2} \leq C\|v\|_{2}.
\end{align*}
(In the final step we used \eqref{eq21a} for $s=2$). \par
In addition, in the course of the proof of the consistency estimates of the fully discrete scheme in Proposition \ref{P31} in section 3 we will need the property that 
if $v \in H^{s}$, $s\geq 3$, is independent of $h$, then
\begin{equation}
\|Pv\|_{3}\leq C_{s}(v),
\label{eq212}
\end{equation}
where $C_{s}(v)$ is a constant depending only on $v$ and $s$. We will also assume that \eqref{eq212}
holds for $P_{0}$ as well, if in addition $v(0)=v(1)=0$. This property does not follow 
from our hypotheses (2.1a-b), \eqref{eq22}. It holds for the Hermite piecewise polynomial 
functions on a
general nonuniform mesh, provided $\mu\geq 2$ (hence, for $r-1\geq 5$, i.e. for at least piecewise
quintic polynomials), cf. \cite{bsw}, and also for smooth splines if $\mu\geq 2$, i.e. for which $r-1\geq 3$, i.e. at least cubic splines. (If $r-1$ is odd, this requires just a quasiuniform mesh,
cf. \cite{s}, while if $r-1$ is even, a uniform mesh guarantees \eqref{eq212} for $\mu\geq 2$,
cf. \cite{bs}.)
\section{The fully discrete scheme and its consistency}
For a positive integer $M$, we let $k=T/M$, $t^{n}=nk$, $n=0,1,\dots,M$, and using the notation
established in Section 2 we let $H(t) = P\eta(t)$, $U(t) = P_{0}u(t)$, $H^{n}=H(t^{n})$,
$U^{n} = U(t^{n})$, where $(\eta,u)$ is the solution of (SW). We also define
\begin{align}
& \Phi = U + HU, \qquad \Phi^{n} = \Phi(t^{n}),
\label{eq31} \\
& F = H_{x} + UU_{x}, \quad F^{n} = F(t^{n}).
\label{eq32}
\end{align}
We discretize in time the ode system represented by the semidiscretization \eqref{eq24}-\eqref{eq25}
by the explicit , fourth-order accurate `classical' Runge-Kutta scheme (RK4), written as follows.
Seek $H_{h}^{n}\in S_{h}$, $U_{h}^{n}\in S_{h,0}$, $0\leq n\leq M$, and $H_{h}^{n,j}\in S_{h}$,
$U_{h}^{n,j}\in S_{h,0}$ for $j=1,2,3$, $0\leq n\leq M-1$, such that
\begin{equation}
\begin{aligned}
& H_{h}^{n,j} - H_{h}^{n} + ka_{j}P\Phi_{hx}^{n,j-1} = 0, \\
& U_{h}^{n,j} - U_{h}^{n} + ka_{j}P_{0}F_{h}^{n,j-1} = 0,
\end{aligned}
\label{eq33}
\end{equation}
for $j=1,2,3$, and
\begin{equation}
\begin{aligned}
& H_{h}^{n+1} - H_{h}^{n}  + kP\biggl[\sum_{j=1}^{4}b_{j}\Phi_{h}^{n,j-1}\biggr]_{x} = 0,\\
& U_{h}^{n+1} - U_{h}^{n} + kP_{0}\biggl[\sum_{j=1}^{4}b_{j}F_{h}^{n,j-1}\biggr] = 0,
\end{aligned}
\label{eq34}
\end{equation}
where
\begin{equation}
\begin{aligned}
& \Phi_{h}^{n,j} = U_{h}^{n,j} + H_{h}^{n,j}U_{h}^{n,j}, \\
& F_{h}^{n,j} = H_{hx}^{n,j} + U_{h}^{n,j}U_{hx}^{n,j},
\end{aligned}
\label{eq35}
\end{equation}
for $j=0,1,2,3$ and
\[
H_{h}^{n,0} = H_{h}^{n}, \quad U_{h}^{n,0} = U_{h}^{n}, \quad
a_{1}=a_{2} = 1/2, \quad a_{3}=1, \quad b_{1}=b_{4}=1/6, \quad b_{2}=b_{3}=1/3,
\]
with
\begin{equation}
H_{h}^{0} = \eta_{h}(0)=P\eta_{0},\quad U_{h}^{0} = u_{h}(0)=P_{0}u_{0},
\label{eq36}
\end{equation}
In order to study the temporal consistency of the scheme \eqref{eq33}-\eqref{eq36} we define the
intermediate stages $V^{n,j}\in S_{h}$, $W^{n,j}\in S_{h,0}$, $j=0,1,2,3$, $0\leq n\leq M-1$, by
the equations
\begin{align}
V^{n,j} - H^{n} + ka_{j}P\Phi_{x}^{n,j-1} & =0,
\label{eq37} \\
W^{n,j} - U^{n} + ka_{j}P_{0}F^{n,j-1} & = 0,
\label{eq38}
\end{align}
and
\[
V^{n,0} = H^{n}, \quad W^{n,0} = U^{n},
\]
where
\begin{align}
\Phi^{n,j} & = W^{n,j} + V^{n,j}W^{n,j},
\label{eq39} \\
F^{n,j} & = V_{x}^{n,j} + W^{n,j}W_{x}^{n,j},
\label{eq310}
\end{align}
for $j=0,1,2,3$, with $\Phi^{n,0} = \Phi^{n}$, $F^{n,0} = F^{n}$. \par
We first estimate the continuous spatial truncation error resulting from replacing 
$\eta$ and $u$ in (SW) by their $L^{2}$ projections on the finite element spaces. 
In the sequel we assume that the solution $(\eta,u)$ of (SW) is sufficiently smooth 
for the purposes of the error estimation.
\begin{lemma} Let $(\eta,u)$ be the solution of (SW) in $[0,T]$. Let $H(t) = P\eta(t)$,
$U(t) = P_{0}u(t)$ and let $\psi(t)\in S_{h}$, $\zeta(t)\in S_{h,0}$, for $0\leq t\leq T$,
be such that
\begin{align}
H_{t} + P\bigl(U + HU\bigr)_{x} & = \psi,
\label{eq311} \\
U_{t} + P_{0}\bigl(H_{x} + UU_{x}\bigr) & = \zeta.
\label{eq312}
\end{align}
Then, for $j=0,1,2,3$, we have
\begin{equation}
\|\partial_{t}^{j}\psi\| + \|\partial_{t}^{j}\zeta\| \leq C h^{r-1},
\label{eq313}
\end{equation}
for $0\leq t\leq T$.
\label{L31}
\end{lemma}
\begin{proof} Subtracting \eqref{eq311} from the equation
$P\bigl(\eta_{t} + u_{x} + (\eta u)_{x}\bigr) = 0$ and defining $\rho:=\eta - H$, $\sigma:= u - U$,
we have
\[
P\sigma_{x} + P\bigl((\eta u)_{x} - (HU)_{x}\bigr) = -\psi.
\]
Since $\eta u - HU = \eta u - (\eta-\rho)(u - \sigma) = \eta\sigma + u\rho - \rho\sigma$,
we see that
\[
P\bigl(\sigma_{x} + (\eta\sigma)_{x} + (u\rho)_{x} - (\rho\sigma)_{x}\bigr) = -\psi,
\]
from which, using the approximation properties of the spaces $S_{h} $, $S_{h,0}$, it follows that
\begin{align*}
\|\partial_{t}^{j}\psi\|  &  \leq \|P\partial_{t}^{j}\sigma_{x}\|
+ \|P\partial_{t}^{j}(\eta\sigma)_{x}\| + \|P\partial_{t}^{j}(u\rho)_{x}\|
+ \|P\partial_{t}^{j}(\rho\sigma)_{x}\| \\
& \leq C (h^{r-1} + h^{2r-1}) \leq C h^{r-1},
\end{align*}
for $j = 0,1,2,3$. Subtracting  \eqref{eq312} from the equation
$P_{0}\bigl(u_{t} + \eta_{x} + uu_{x}\bigr) = 0$ gives
\[
P_{0}\rho_{x} + P_{0}(uu_{x} - UU_{x}) = -\zeta.
\]
Since
$uu_{x} - UU_{x} = uu_{x} - (u - \sigma)(u_{x} - \sigma_{x})
= (u\sigma)_{x} - \sigma\sigma_{x}$,
it follows, for $j=0,1,2,3$, that
\begin{align*}
\|\partial_{t}^{j}\zeta\|  & \leq \|P_{0}\partial_{t}^{j}\rho_{x}\|
+ \|P_{0}\partial_{t}^{j}(u\sigma)_{x}\| + \|P_{0}\partial_{t}^{j}(\sigma\sigma_{x})\| \\
& \leq C (h^{r-1} + h^{2r-1}) \leq C h^{r-1},
\end{align*}
and \eqref{eq313} is proved.
\end{proof}
In the following proposition we estimate appropriately defined local errors of the fully discrete scheme \eqref{eq33}-\eqref{eq36}. The local errors $\delta_{1}^{n}\in S_{h}$,
$\delta_{2}^{n}\in S_{h,0}$ are expressed in terms of the $L^{2}$ projections $H^{n}=P\eta(t^{n})$,
$U^{n} = P_{0}u(t^{n})$, and the quantities $\Phi^{n,i}$, $F^{n,i}$ defined by \eqref{eq39},
\eqref{eq310} as nonlinear functions of the intermediate stages 
$V^{n,i}$, $W^{n,i}$ of a single step of the RK4 scheme with starting values 
$H^{n}$, $U^{n}$, cf. \eqref{eq37}, \eqref{eq38}. The plan of the error estimation is straightforward but the details of the proof are rather technical.
We find expansions of $\Phi^{n,i}$, $F^{n,i}$, for $i=1,2,3$ in powers of $k$ up to terms of
$O(k^{2})$ for $i=1$ and up to terms of $O(k^{3})$ for $i=2$ and 3, and we estimate the remainders by bounds of $O(h^{r-1} + k^{4})$ in appropriate norms. The constants in these 
error bounds depend polynomially on the Courant number $\lambda=k/h$. The expressions for $\Phi^{n,i}$, $F^{n,i}$ are combined as in the final step of the RK4 scheme to yield the required estimates of the local errors
after cancellation of the lower-order terms. \par
In bounding the remainders of $\Phi^{n,i}$, $F^{n,i}$ for $i=1,2$, use is made of the 
standard approximation and inverse properties (2.1), \eqref{eq22} of the finite element 
spaces, in particular of the stability and approximation estimates of the $L^{2}$ 
projections that follow from these properties. But in the course of bounding some terms 
of the remainders of $\Phi^{n,3}$ and $F^{n,3}$ we need to find $L^{2}$ bounds independent 
of $h$ of third-order spatial derivatives of $H_{t}$ and $U_{t}$; for this purpose we use the hypothesis that \eqref{eq212} holds.
\begin{proposition} Let $(\eta, u)$ be the solution of (SW) and let $\lambda = k / h$.
If $\delta_{1}^{n}$, $\delta_{2}^{n}$, for $0\leq n\leq M-1$, are such that
\begin{align}
\delta_{1}^{n} & = H^{n+1} - H^{n}  + k P\bigl[\sum_{j=1}^{4}b_{j}\Phi^{n,j-1}\bigr]_{x},
\label{eq314} \\
\delta_{2}^{n} & = U^{n+1} - U^{n}  + k P_{0}\bigl[\sum_{j=1}^{4}b_{j}F^{n,j-1}\bigr],
\label{eq315}
\end{align}
then, there exists a constant $C_{\lambda}$ that depends polynomially on $\lambda$ such that
\[
\max_{0\leq n\leq M-1}(\|\delta_{1}^{n}\| + \|\delta_{2}^{n}\|) \leq C_{\lambda}k(k^{4} + h^{r-1}).
\]
\label{P31}
\end{proposition}
\begin{proof} From \eqref{eq37} it follows that
\[
V^{n,1} - H^{n} + a_{1} k P\Phi_{x}^{n} = 0.
\]
Hence, by \eqref{eq311}, \eqref{eq31},
\begin{equation}
V^{n,1} = H^{n} + a_{1}kH_{t}^{n} - a_{1}k\psi^{n}.
\label{eq316}
\end{equation}
In addition, by \eqref{eq38} and \eqref{eq312} we have
\[
W^{n,1} - U^{n} + a_{1}kP_{0}F^{n} = 0,
\]
and consequently, from \eqref{eq312}, \eqref{eq32},
\begin{equation}
W^{n,1} = U^{n} + a_{1}kU_{t}^{n} - a_{1}k\zeta^{n}.
\label{eq317}
\end{equation}
So, by  \eqref{eq316} and \eqref{eq317},
\begin{equation}
V^{n,1}W^{n,1} = H^{n}U^{n} + a_{1}k(HU)_{t}^{n} + a_{1}^{2}k^{2}H_{t}^{n}U_{t}^{n} + v_{1}^{n},
\label{eq318}
\end{equation}
where
\[
v_{1}^{n} = -a_{1}k(H^{n}\zeta^{n} + U^{n}\psi^{n}) - a_{1}^{2}k^{2}(H_{t}^{n}\zeta^{n}
+U_{t}^{n}\psi^{n} - \psi^{n}\zeta^{n}).
\]
Therefore, by \eqref{eq317}, \eqref{eq318}, we obtain
\[
W^{n,1} + V^{n,1}W^{n,1} = U^{n} + H^{n}U^{n} + a_{1}k\bigl(U_{t}^{n} + (HU)_{t}^{n}\bigr)
+ a_{1}^{2}k^{2}H_{t}^{n}U_{t}^{n} + v_{2}^{n},
\]
so that, the desired expansion of $\Phi^{n,1}$ in powers of $k$ is given by
\begin{equation}
\Phi^{n,1} = \Phi^{n} + a_{1}k\Phi_{t}^{n} + a_{1}^{2}k^{2}H_{t}^{n}U_{t}^{n} + v_{2}^{n},
\label{eq319}
\end{equation}
where
\[
v_{2}^{n} = -a_{1}k\zeta^{n} + v_{1}^{n}.
\]
Note that by \eqref{eq313} and \eqref{eq22}, (2.3) it follows that
\begin{equation}
\|v_{2}^{n}\|_{1} \leq C_{\lambda} h^{r-1},
\label{eq320}
\end{equation}
where $C_{\lambda}$ is a first-order polynomial in $\lambda$ with positive coefficients. 
(In the sequel we shall denote by $C_{\lambda}$ polynomials of $\lambda$ with positive coefficients without reference to their degree.) In deriving \eqref{eq320} we made use 
of the fact (something that we will also do in the sequel, without explicit mention,) 
that the quantities $\|\partial_{t}^{i}H^{n}\|_{j}$,
$\|\partial_{t}^{i}U^{n}\|_{j}$ for $j=0,1,2$ and for each $i$, are bounded, uniformly in $n$, by constants independent of the discretization parameters $k$ and $h$. This follows 
from (2.3a), \eqref{eq211} and the smoothness of $\eta$ and $u$. The same holds for the quantities
$\|\partial_{t}^{i}H^{n}\|_{j,\infty}$, $\|\partial_{t}^{i}U^{n}\|_{j,\infty}$ for $j=0,1$, 
as seen from \eqref{eq23b} and \eqref{eq21b}.  \\ \noindent
Now, from \eqref{eq317}
\begin{equation}
W^{n,1}W_{x}^{n,1} = U^{n}U_{x}^{n} + a_{1}k(UU_{x})_{t}^{n} +a_{1}^{2}k^{2}U_{t}^{n}U_{tx}^{n}
+ w_{1}^{n},
\label{eq321}
\end{equation}
where
\[
w_{1}^{n} = -a_{1}k(U^{n}\zeta^{n})_{x} - a_{1}^{2}k^{2}(U_{t}^{n}\zeta^{n})_{x}
+ a_{1}^{2}k^{2}\zeta^{n}\zeta_{x}^{n}.
\]
Hence, from \eqref{eq316}, \eqref{eq321}, it follows that
\[
V_{x}^{n,1} + W^{n,1}W_{x}^{n,1} = H_{x}^{n} + U^{n}U_{x}^{n}
+a_{1}k (H_{x} + UU_{x})_{t}^{n} + a_{1}^{2}k^{2}U_{t}^{n}U_{tx}^{n} + w_{2}^{n},
\]
i.e.
\begin{equation}
F^{n,1} = F^{n} + a_{1}k F_{t}^{n} + a_{1}^{2}k^{2} U_{t}^{n}U_{tx}^{n} + w_{2}^{n},
\label{eq322}
\end{equation}
which is the required expansion of $F^{n,1}$. In the above
\[
w_{2}^{n} = - a_{1}k\psi_{x}^{n} + w_{1}^{n},
\]
for which, using \eqref{eq313}, the inverse inequalities, and the remarks following \eqref{eq320},
we obtain the estimate
\begin{equation}
\|w_{2}^{n}\| \leq C_{\lambda} h^{r-1}.
\label{eq323}
\end{equation}
We now find the expansions of $\Phi^{n,2}$ and $F^{n,2}$ up to $O(k^{3})$ terms. From \eqref{eq37},
i.e.
\[
V^{n,2} = H^{n} - a_{2}kP\Phi_{x}^{n,1},
\]
it follows, in view of \eqref{eq319}, that
\[
V^{n,2} = H^{n} - a_{2}kP\Phi_{x}^{n} - a_{1}a_{2}k^{2}P\Phi_{tx}^{n}
 - a_{1}^{2}a_{2}k^{3}P(H_{t}^{n}U_{t}^{n})_{x} - a_{2}kPv_{2x}^{n},
\]
which, in view of \eqref{eq311} gives
\begin{equation}
V^{n,2} = H^{n} + a_{2}kH_{t}^{n} + a_{1}a_{2}k^{2} H_{tt}^{n}
- a_{1}^{2}a_{2}k^{3}P(H_{t}^{n}U_{t}^{n})_{x}  + \psi_{1}^{n},
\label{eq324}
\end{equation}
where
\[
\psi_{1}^{n} = - a_{2}k\psi^{n} - a_{1}a_{2}k^{2} \psi_{t}^{n} - a_{2}kPv_{2x}^{n}.
\]
From \eqref{eq320} and \eqref{eq313} it follows that
\begin{equation}
\|\psi_{1}^{n}\| \leq C_{\lambda}k h^{r-1},
\label{eq325}
\end{equation}
and, by the inverse properties, that
\begin{equation}
\|\psi_{1x}^{n}\|\leq C_{\lambda} h^{r-1}.
\label{eq326}
\end{equation}
Moreover, since from \eqref{eq38}
\[
W^{n,2} = U^{n} - a_{2}kP_{0}F^{n,1},
\]
we obtain, using \eqref{eq322},
\[
W^{n,2} = U^{n} - a_{2}P_{0}F^{n} - a_{1}a_{2}k^{2}P_{0}F_{t}^{n}
- a_{1}^{2}a_{2}k^{3}P_{0}(U_{t}^{n}U_{tx}^{n}) - a_{2}kP_{0}w_{2}^{n},
\]
and finally, from \eqref{eq312},
\begin{equation}
W^{n,2} = U^{n} + a_{2}kU_{t}^{n} + a_{1}a_{2}k^{2}U_{tt}^{n}
- a_{1}^{2}a_{2}k^{3}P_{0}(U_{t}^{n}U_{tx}^{n}) + \zeta_{1}^{n},
\label{eq327}
\end{equation}
where
\[
\zeta_{1}^{n} = - a_{2}k\zeta^{n} - a_{1}a_{2}k^{2}\zeta_{t}^{n} - a_{2}kP_{0}w_{2}^{n},
\]
that we estimate, using \eqref{eq313} and \eqref{eq323}, by
\begin{equation}
\|\zeta_{1}^{n} \|\leq C_{\lambda}kh^{r-1},
\label{eq328}
\end{equation}
and, using the inverse properties, by
\begin{equation}
\|\zeta_{1x}^{n}\| \leq C_{\lambda} h^{r-1}.
\label{eq329}
\end{equation}
Now, from \eqref{eq324}, \eqref{eq327}, taking into account that $a_{1}=a_{2}$ we have
\begin{equation}
\begin{aligned}
V^{n,2}W^{n,2} & = H^{n}U^{n} + a_{2}k(HU)_{t}^{n}
+ a_{1}a_{2}k^{2}(H^{n}U_{tt}^{n} +H_{t}^{n}U_{t}^{n} +  H_{tt}^{n}U^{n})
+ a_{1}a_{2}^{2}k^{3}(H_{t}^{n}U_{tt}^{n} + H_{tt}^{n}U_{t}^{n}) \\
& \,\,\,\,\,\,\,
- a_{1}^{2}a_{2}k^{3}\bigl(H^{n}P_{0}(U_{t}^{n}U_{tx}^{n}) +U^{n} P(H_{t}^{n}U_{t}^{n})_{x}\bigr)
+ v_{3}^{n},
\end{aligned}
\label{eq330}
\end{equation}
where
\begin{align*}
v_{3}^{n} & = H^{n}\zeta_{1}^{n} + a_{2}kH_{t}^{n}\zeta_{1}^{n}
+ a_{1}a_{2}k^{2}H_{tt}^{n}\zeta_{1}^{n}  - a_{1}^{2}a_{2}k^{3}P(H_{t}^{n}U_{t}^{n})_{x}\zeta_{1}^{n}
- a_{1}^{2}a_{2}^{2}k^{4}\bigl(H_{t}^{n}P_{0}(U_{t}^{n}U_{tx}^{n}) + U_{t}^{n}P(H_{t}^{n}U_{t}^{n})_{x}\bigr) \\
& \,\,\,\,\,\, + a_{1}^{2}a_{2}^{2}k^{4}H_{tt}^{n}U_{tt}^{n}
- a_{1}^{3}a_{2}^{2}k^{5}\bigl(H_{tt}^{n}P_{0}(U_{t}^{n}U_{tx}^{n}) + U_{tt}^{n}P(H_{t}^{n}U_{t}^{n})_{x}\bigr)
+ a_{1}^{4}a_{2}^{2}k^{6}P(H_{t}^{n}U_{t}^{n})_{x}P_{0}(H_{t}^{n}U_{tx}^{n}) \\
&\,\,\,\,\,\, + \psi_{1}^{n}W^{n,2}.
\end{align*}
Using \eqref{eq328}, \eqref{eq329}, \eqref{eq325}, \eqref{eq313}, and taking into account the
inverse inequalities and the remarks following \eqref{eq320} we may estimate $v_{3}^{n}$ as
follows:
\begin{equation}
\|v_{3}^{n}\| \leq C_{\lambda}kh^{r-1} + Ck^{4}\, \quad \text{and} \quad
\|v_{3}^{n}\|_{1} \leq C_{\lambda}h^{r-1} + Ck^{4}.
\label{eq331}
\end{equation}
Finally, writing \eqref{eq330} in the form
\begin{align*}
V^{n,2}W^{n,2} & = H^{n}U^{n} + a_{2}k(HU)_{t}^{n} + a_{1}a_{2}k^{2}(HU)_{tt}^{n}
- a_{1}a_{2}k^{2}H_{t}^{n}U_{t}^{n}  + a_{1}a_{2}^{2}k^{3}(H_{t}U_{t})_{t}^{n} \\
& \,\,\,\,\,\, - a_{1}^{2}a_{2}k^{3}\bigl(H^{n}P_{0}(U_{t}^{n}U_{tx}^{n}) + U^{n}P(H_{t}^{n}U_{t}^{n})_{x}\bigr) + v_{3}^{n},
\end{align*}
we obtain, using \eqref{eq39} and \eqref{eq327}, the desired expansion of $\Phi^{n,2}$ in powers
of $k$ :
\begin{equation}
\begin{aligned}
\Phi^{n,2} & = \Phi^{n} + a_{2}k\Phi_{t}^{n} + a_{1}a_{2}k^{2}\Phi_{tt}^{n} - a_{1}a_{2}k^{2}H_{t}^{n}U_{t}^{n}
+ a_{1}a_{2}^{2}k^{3}(H_{t}U_{t})_{t}^{n} \\
&\,\,\,\,\,\, - a_{1}^{2}a_{2}k^{3}\bigl[(1+H^{n})P_{0}(U_{t}^{n}U_{tx}^{n}) + U^{n}P(H_{t}^{n}U_{t}^{n})_{x}\bigr] +  v_{4}^{n},
\end{aligned}
\label{eq332}
\end{equation}
where $v_{4}^{n} = \zeta_{1}^{n} + v_{3}^{n}$. Hence, from \eqref{eq328}, \eqref{eq329}, and
\eqref{eq331} we have
\begin{equation}
\|v_{4}^{n}\| \leq C_{\lambda} k h^{r-1}, \quad \|v_{4x}^{n}\|\leq C_{\lambda}h^{r-1} + Ck^{4}.
\label{eq333}
\end{equation}
Now it follows from \eqref{eq327} that
\begin{align*}
W^{n,2}W_{x}^{n,2} & = \biggl[U^{n} + a_{2}kU_{t}^{n} + a_{1}a_{2}k^{2}U_{tt}^{n}
- a_{1}^{2}a_{2}k^{3}P_{0}(U_{t}^{n}U_{tx}^{n}) + \zeta_{1}^{n}\biggr]\cdot \\
& \,\,\,\,\,\,\,\biggl[U_{x}^{n} + a_{2}kU_{tx}^{n} + a_{1}a_{2}k^{2}U_{ttx}^{n}
- a_{1}^{2}a_{2}k^{3}\bigl(P_{0}(U_{t}^{n}U_{tx}^{n})\bigr)_{x} + \zeta_{1x}^{n}\biggr],
\end{align*}
and, consequently, since $a_{1}=a_{2}$, that
\begin{equation}
\begin{aligned}
W^{n,2}W_{x}^{n,2} & = U^{n}U_{x}^{n} + a_{2}k(UU_{x})_{t}^{n} + a_{1}a_{2}k^{2}(UU_{x})_{tt}^{n}
- a_{1}a_{2}k^{2}U_{t}^{n}U_{tx}^{n} + a_{1}a_{2}^{2}k^{3}(U_{t}^{n}U_{tt}^{n})_{x} \\
&\,\,\,\,\,\, - a_{1}^{2}a_{2}k^{3}\bigl[U^{n}P_{0}(U_{t}^{n}U_{tx}^{n})\bigr]_{x} + w_{3}^{n},
\end{aligned}
\label{eq334}
\end{equation}
in which, using the approximation and inverse properties of the finite element spaces, \eqref{eq329},
and observations like the ones following \eqref{eq320}, we may estimate $w_{3}^{n}$ by the
inequalities
\begin{equation}
\|w_{3}^{n}\| \leq C_{\lambda}h^{r-1} + Ck^{4}, \quad
k\|w_{3x}^{n}\|\leq C_{\lambda}(h^{r-1} + k^{4}).
\label{eq335}
\end{equation}
Now, the definition of $F^{n,2}$ in \eqref{eq310}, \eqref{eq324}, and \eqref{eq334} give
\begin{equation}
\begin{aligned}
F^{n,2} & = F^{n} + a_{2}kF_{t}^{n} + a_{1}a_{2}k^{2}F_{tt}^{n} -a_{1}a_{2}k^{2}U_{t}^{n}U_{tx}^{n}
+ a_{1}a_{2}^{2}k^{3}(U_{t}^{n}U_{tt}^{n})_{x}\\
&\,\,\,\,\,\, - a_{1}^{2}a_{2}k^{3}\bigl[P(H_{t}^{n}U_{t}^{n})_{x}
+ U^{n}P_{0}(U_{t}^{n}U_{tx}^{n})\bigr]_{x} + w_{4}^{n},
\end{aligned}
\label{eq336}
\end{equation}
where
\[
w_{4}^{n} = \psi_{1x}^{n} + w_{3}^{n}.
\]
From \eqref{eq326} and \eqref{eq335} it follows that
\begin{equation}
\|w_{4}^{n}\| \leq C_{\lambda}h^{r-1} + Ck^{4}.
\label{eq337}
\end{equation}
This completes the required expansion of $F^{n,2}$ in powers of $k$. \par 
We now compute the required expansions of $\Phi^{n,3}$ and $F^{n,3}$ up to $O(k^{3})$ terms. 
In the course of estimating some of the $O(k^{4})$ remainder terms we need to find $L^{2}$ bounds independent of $h$ of third-order spatial derivatives of $U_{t}$ and $H_{t}$ and for this purpose we need the hypothesis \eqref{eq212}. Since
\[
V^{n,3} = H^{n} - a_{3}kP\Phi_{x}^{n,2} = H^{n} - kP\Phi_{x}^{n,2},
\]
from \eqref{eq332} and \eqref{eq311} it follows that
\begin{align*}
V^{n,3} & = H^{n} + kH_{t}^{n} - k\psi^{n} + a_{2}k^{2}H_{tt}^{n} - a_{2}k^{2}\psi_{t}^{n}
+ a_{1}a_{2}k^{3}H_{ttt}^{n} - a_{1}a_{2}k^{3}\psi_{tt}^{n}
+ a_{1}a_{2}k^{3}P(H_{t}^{n}U_{t}^{n})_{x} \\
&\,\,\,\,\,\, - a_{1}a_{2}^{2}k^{4}P\bigl((H_{t}U_{t})_{tx}^{n}\bigr)
+ a_{1}^{2}a_{2}k^{4}P\bigl[(1+H^{n})P_{0}(U_{t}^{n}U_{tx}^{n}) + U^{n}P(H_{t}^{n}U_{t}^{n})_{x}\bigr]_{x} - kPv_{4x}^{n},
\end{align*}
which we write as
\begin{equation}
V^{n,3} = H^{n} + kH_{t}^{n} + a_{2}k^{2}H_{tt}^{n} + a_{1}a_{2}k^{3}H_{ttt}^{n}
+ a_{1}a_{2}k^{3}P(H_{t}^{n}U_{t}^{n})_{x} + \psi_{2}^{n},
\label{eq338}
\end{equation}
where
\begin{align*}
\psi_{2}^{n} & = - k\psi^{n} - a_{2}k^{2}\psi_{t}^{n} - a_{1}a_{2}k^{3}\psi_{tt}^{n}
- a_{1}a_{2}^{2}k^{4}P\bigl((H_{t}U_{t})_{tx}^{n}\bigr) \\
&\,\,\,\,\,\, + a_{1}^{2}a_{2}k^{4}P\bigl[(1+H^{n})P_{0}(U_{t}^{n}U_{tx}^{n}) + U^{n}P(H_{t}^{n}U_{t}^{n})_{x}\bigr]_{x} - kPv_{4x}^{n}.
\end{align*}
From \eqref{eq333}, the approximation and inverse properties of the finite element 
spaces, and the remarks following \eqref{eq320} we get
\begin{equation}
\|\psi_{2}^{n}\| \leq C_{\lambda}kh^{r-1} + Ck^{4}.
\label{eq339}
\end{equation}
By similar considerations and using also the hypothesis \eqref{eq212} we infer in addition that
\begin{equation}
\|\psi_{2x}^{n}\| \leq C_{\lambda}(h^{r-1} + k^{4}).
\label{eq340}
\end{equation}
Also, since
\[
W^{n,3} = U^{n} - a_{3}kP_{0}F^{n,2} = U^{n} - kP_{0}F^{n,2},
\]
from \eqref{eq336} and \eqref{eq312} we obtain
\begin{align*}
W^{n,3} & = U^{n} + kU_{t}^{n} - k\zeta^{n} + a_{2}k^{2}U_{tt}^{n} - a_{2}k^{2}\zeta_{t}^{n}
+ a_{1}a_{2}k^{3}U_{ttt}^{n} - a_{1}a_{2}k^{3}\zeta_{tt}^{n}
+ a_{1}a_{2}k^{3}P_{0}(U_{t}^{n}U_{tx}^{n}) \\
&\,\,\,\,\,\, + a_{1}^{2}a_{2}k^{4}P_{0}\bigl[P(H_{t}^{n}U_{t}^{n})_{x} - U_{t}^{n}U_{tt}^{n}
+ U^{n}P_{0}(U_{t}^{n}U_{tx}^{n})\bigr]_{x} - kP_{0}w_{4}^{n},
\end{align*}
i.e.
\begin{equation}
W^{n,3} = U^{n} + kU_{t}^{n} + a_{2}k^{2}U_{tt}^{n} + a_{1}a_{2}k^{3}U_{ttt}^{n}
+ a_{1}a_{2}k^{3}P_{0}(U_{t}^{n}U_{tx}^{n}) + \zeta_{2}^{n},
\label{eq341}
\end{equation}
where
\[
\zeta_{2}^{n} = -k\zeta^{n} - a_{1}k^{2}\zeta_{t}^{n} - a_{1}a_{2}k^{3}\zeta_{tt}^{n}
 + a_{1}^{2}a_{2}k^{4}P_{0}\bigl[P(H_{t}^{n}U_{t}^{n})_{x} - U_{t}^{n}U_{tt}^{n}
+ U^{n}P_{0}(U_{t}^{n}U_{tx}^{n})\bigr]_{x} - kP_{0}w_{4}^{n},
\]
From \eqref{eq313}, \eqref{eq337}, the approximation and inverse properties of the 
finite element spaces, and the remarks following \eqref{eq320} we may see that
\begin{equation}
\|\zeta_{2}^{n}\| \leq C_{\lambda}kh^{r-1} + Ck^{4}.
\label{eq342}
\end{equation}
By similar considerations and also the hypothesis \eqref{eq212} it follows that
\begin{equation}
\|\zeta_{2x}^{n}\| \leq C_{\lambda}(h^{r-1} + k^{4}).
\label{eq343}
\end{equation}
From \eqref{eq338} and \eqref{eq341}, since $a_{2}=1/2=a_{1}$, we see that
\begin{equation}
\begin{aligned}
V^{n,3}W^{n,3} & = H^{n}U^{n} + k(HU)_{t}^{n} + a_{2}k^{2}(HU)_{tt}^{n} + a_{1}a_{2}k^{3}(HU)_{ttt}^{n}
- a_{1}a_{2}k^{3}(H_{t}U_{t})_{t}^{n}\\
&\,\,\,\,\,\,
+ a_{1}a_{2}k^{3}\bigl[H^{n}P_{0}(U_{t}^{n}U_{tx}^{n}) + U^{n}P(H_{t}^{n}U_{t}^{n})_{x}\bigr]
+ v_{5}^{n}.
\end{aligned}
\label{eq344}
\end{equation}
From \eqref{eq313}, \eqref{eq342}, \eqref{eq339}, \eqref{eq340}, the approximation 
and inverse properties of the finite element spaces, and the remarks following 
\eqref{eq320} it follows for the remainder term in \eqref{eq344} that
\begin{equation}
\|v_{5}^{n}\| \leq C_{\lambda}kh^{r-1} + Ck^{4} + C_{\lambda}k^{8}.
\label{eq345}
\end{equation}
Similar considerations and in addition \eqref{eq343}, \eqref{eq340}, lead to
\begin{equation}
\|v_{5x}^{n}\| \leq C_{\lambda}(h^{r-1} + k^{4}).
\label{eq346}
\end{equation}
Hence, from \eqref{eq341}. \eqref{eq39}, \eqref{eq344} we have the desired expansion of $\Phi^{n,3}$
given by
\begin{equation}
\begin{aligned}
\Phi^{n,3} & = \Phi^{n} + k \Phi_{t}^{n} + a_{2}k^{2}\Phi_{tt}^{n} + a_{1}a_{2}k^{3}\Phi_{ttt}^{n}
- a_{1}a_{2}k^{3}(H_{t}U_{t})_{t}^{n} \\
&\,\,\,\,\,\,
+ a_{1}a_{2}k^{3}\bigl[(1+H^{n})P_{0}(U_{t}^{n}U_{tx}^{n}) + U^{n}P(H_{t}^{n}U_{t}^{n})_{x}]
+ v_{6}^{n},
\end{aligned}
\label{eq347}
\end{equation}
in which $v_{6}^{n} = \psi_{2}^{n} + v_{5}^{n}$. Hence, from  \eqref{eq339}, \eqref{eq340},
\eqref{eq345}, \eqref{eq346}, it follows that
\begin{equation}
\|v_{6}^{n}\| \leq C_{\lambda}kh^{r-1} + Ck^{4}, \quad
\|v_{6x}^{n}\| \leq C_{\lambda}(h^{r-1} + k^{4}).
\label{eq348}
\end{equation}
From \eqref{eq341} we see that
\begin{align*}
W^{n,3}W_{x}^{n,3} & = \biggl[U^{n} + kU_{t}^{n} + a_{2}k^{2}U_{tt}^{n} + a_{1}a_{2}k^{3}U_{ttt}^{n}
+ a_{1}a_{2}k^{3}P_{0}(U_{t}^{n}U_{tx}^{n}) + \zeta_{2}^{n}\biggr]\cdot \\
& \,\,\,\,\,\,\,\biggl[U_{x}^{n} + kU_{tx}^{n} + a_{2}k^{2}U_{ttx}^{n} + a_{1}a_{2}k^{3}U_{tttx}^{n}
+ a_{1}a_{2}k^{3}\bigl(P_{0}(U_{t}^{n}U_{tx}^{n})\bigr)_{x} + \zeta_{2x}^{n}\biggr],
\end{align*}
and therefore, using the fact that $a_{1}=1/2$,
\begin{equation}
\begin{aligned}
W^{n,3}W_{x}^{n,3} & = U^{n}U_{x}^{n} + k(UU_{x})_{t}^{n} + a_{2}k^{2}(UU_{x})_{tt}^{n}
+ a_{1}a_{2}k^{3}(UU_{x})_{ttt}^{n} \\
&\,\,\,\,\,\, - a_{1}a_{2}k^{3}(U_{t}U_{tx})_{t}^{n} + a_{1}a_{2}k^{3}\bigl[U^{n}P_{0}(U_{t}^{n}U_{tx}^{n})\bigr]_{x} + w_{5}^{n}.
\end{aligned}
\label{eq349}
\end{equation}
For the remainder term, using \eqref{eq342}, \eqref{eq343}, the approximation and inverse properties of the finite element spaces, and considerations such as the ones following \eqref{eq320} we get
\begin{equation}
\|w_{5}^{n}\| \leq C_{\lambda}(h^{r-1} + k^{4}).
\label{eq350}
\end{equation}
Finally, from \eqref{eq338}, \eqref{eq310}, and \eqref{eq349}, we have the required 
expansion of $F^{n,3}$
\begin{equation}
\begin{aligned}
F^{n,3} & = F^{n} + kF_{t}^{n} + a_{2}k^{2}F_{tt}^{n} + a_{1}a_{2}k^{3}F_{ttt}^{n}
- a_{1}a_{2}k^{3}(U_{t}U_{tt})_{x}^{n} \\
& \,\,\,\,\,\,
+ a_{1}a_{2}k^{3}\bigl[P(H_{t}^{n}U_{t}^{n})_{x} 
+ U^{n}P_{0}(U_{t}^{n}U_{tx}^{n})\bigr]_{x} + w_{6}^{n},
\end{aligned}
\label{eq351}
\end{equation}
where  $w_{6}^{n} = \psi_{2,x}^{n} + w_{5}^{n}$. Therefore, by \eqref{eq340} 
and \eqref{eq351} we conclude that
\begin{equation}
\|w_{6}^{n}\| \leq C_{\lambda}(h^{r-1} + k^{4}).
\label{eq352}
\end{equation}
\indent 
We now come to the final line of the RK4 algorithm and the expansions of
$\sum_{j=1}^{4}b_{j}\Phi^{n,j-1}$, $\sum_{j=1}^{4}b_{j}F^{n,j-1}$ that are needed in the expressions \eqref{eq314}, \eqref{eq315} of the local errors. Since $\sum_{j=1}^{4}b_{j}=1$, from \eqref{eq319}, \eqref{eq332}, \eqref{eq347}, it follows that
\begin{align*}
b_{1}\Phi^{n} & + b_{2}\Phi^{n,1} + b_{3}\Phi^{n,2} + b_{4}\Phi^{n,3}  =
\Phi^{n} + (b_{2}a_{1}+b_{3}a_{2} + b_{4})k\Phi_{t}^{n} + (b_{2}a_{1}^{2}
- b_{3}a_{1}a_{2})k^{2}H_{t}^{n}U_{t}^{n}\\
& + (b_{3}a_{1}a_{2} + b_{4}a_{2})k^{2}\Phi_{tt}^{n} +(b_{3}a_{1}a_{2}^{2}
- b_{4}a_{1}a_{2})k^{3}(H_{t}U_{t})_{t}^{n}\\
& + (- b_{3}a_{1}^{2}a_{2} + b_{4}a_{1}a_{2})k^{3}
\bigl[(1+H^{n})P_{0}(U_{t}^{n}U_{tx}^{n}) + U^{n}P(H_{t}^{n}U_{t}^{n})_{x}\bigr]\\
& + b_{4}a_{1}a_{2}k^{3}\Phi_{ttt}^{n} + b_{2}v_{2}^{n} + b_{3}v_{4}^{n} + b_{4}v_{6}^{n},
\end{align*}
and, therefore, using the values of $b_{1}$, $b_{2}$, $b_{3}$, $b_{4}$, $a_{1}$, $a_{2}$,
\[
b_{1}\Phi^{n} + b_{2}\Phi^{n,1} + b_{3}\Phi^{n,2} + b_{4}\Phi^{n,3} =
\Phi^{n} + \tfrac{k}{2}\Phi_{t}^{n} + \tfrac{k^{2}}{6}\Phi_{tt}^{n} + \tfrac{k^{3}}{24}\Phi_{ttt}^{n}
+ \tfrac{1}{3}v_{2}^{n} + \tfrac{1}{3}v_{4}^{n} + \tfrac{1}{6}v_{6}^{n}.
\]
From \eqref{eq31}, \eqref{eq311}, \eqref{eq314}, and the above equality, we see that
\[
\delta_{1}^{n} = H^{n+1} - H^{n} -kH_{t}^{n} - \tfrac{k^{2}}{2}H_{tt}^{n} - \tfrac{k^{3}}{6}H_{ttt}^{n}
- \tfrac{k^{4}}{24}H_{tttt}^{n} + \alpha^{n},
\]
where
\[
\alpha^{n} = k\psi^{n} + \tfrac{k^{2}}{2}\psi_{t}^{n} + \tfrac{k^{3}}{6}\psi_{tt}^{n}
+ \tfrac{k^{4}}{24}\psi_{ttt}^{n} + \tfrac{k}{3}(Pv_{2x}^{n} + Pv_{4x}^{n}) + \tfrac{k}{6}Pv_{6x}^{n}.
\]
Therefore, since by \eqref{eq313}, \eqref{eq320}, \eqref{eq333}, \eqref{eq348},
\[
\|\alpha^{n}\| \leq C_{\lambda}k(h^{r-1} + k^{4}),
\]
it follows by Taylor's theorem that
\begin{equation}
\|\delta_{1}^{n}\| \leq C_{\lambda}k(h^{r-1} + k^{4}).
\label{eq353}
\end{equation}
In addition from \eqref{eq322}, \eqref{eq336}, \eqref{eq351} we obtain
\begin{align*}
b_{1}F^{n} & + b_{2}F^{n,1} + b_{3}F^{n,2} + b_{4}F^{n,3} =
F^{n} + (b_{2}a_{1} + b_{3}a_{2} + b_{4})kF_{t}^{n} + (b_{2}a_{1}^{2}
- b_{3}a_{1}a_{2})k^{2}U_{t}^{n}U_{tx}^{n} \\
& + (b_{3}a_{1}a_{2} + b_{4}a_{2})k^{2}F_{tt}^{n}
+ (b_{3}a_{1}a_{2}^{2} - b_{4}a_{1}a_{2})k^{3}(U_{t}^{n}U_{tt}^{n})_{x}\\
& + (-b_{3}a_{1}^{2}a_{2} + b_{4}a_{1}a_{2})k^{3}
\bigl[P(H_{t}^{n}U_{t}^{n}) + U^{n}P_{0}(U_{t}^{n}U_{tx}^{n})\bigr]_{x}\\
& + b_{4}a_{1}a_{2}k^{3}F_{ttt}^{n} + b_{2}w_{2}^{n} + b_{3}w_{4}^{n} + b_{4}w_{6}^{n},
\end{align*}
and therefore
\[
b_{1}F^{n} + b_{2}F^{n,1} + b_{3}F^{n,2} + b_{4}F^{n,3} =F^{n} + \tfrac{k}{2}F_{t}^{n}
+ \tfrac{k^{2}}{6}F_{tt}^{n} + \tfrac{k^{3}}{24}F_{ttt}^{n} + \tfrac{1}{3}w_{2}^{n}
+ \tfrac{1}{3}w_{4}^{n} + \tfrac{1}{6}w_{6}^{n}.
\]
From \eqref{eq32}, \eqref{eq312}, \eqref{eq315}, and the above, it follows that
\[
\delta_{2}^{n} = U^{n+1} - U^{n} - kU_{t}^{n} - \tfrac{k^{2}}{2}U_{tt}^{n}
- \tfrac{k^{3}}{6}U_{ttt}^{n} - \tfrac{k^{4}}{24}U_{tttt}^{n} + \beta^{n},
\]
where
\[
\beta^{n} = k\zeta^{n} + \tfrac{k^{2}}{2}\zeta_{t}^{n} + \tfrac{k^{3}}{6}\zeta_{tt}^{n}
+ \tfrac{k^{4}}{24}\zeta_{ttt}^{n} + \tfrac{k}{3}(P_{0}w_{2}^{n}+ P_{0}w_{4}^{n})
+ \tfrac{k}{6}P_{0}w_{6}^{n},
\]
and, in view of \eqref{eq323}, \eqref{eq337}, and \eqref{eq352},
\[
\|\beta^{n}\| \leq C_{\lambda}k(h^{r-1} + k^{4}).
\]
By the above and Taylor's theorem we see that $\|\delta_{2}^{n}\| \leq C_{\lambda}k(h^{r-1} + k^{4})$.
This estimate and \eqref{eq353} conclude the proof of the proposition.
\end{proof}
\section{error estimate}
In this section we analyze the convergence of the fully discrete scheme \eqref{eq33}-\eqref{eq36} to the solution of (SW) in the $L^{2}\times L^{2}$ norm. 
We start with three technical Lemmata
whose notation and results will be used in the course of proof of the error 
estimate in Theorem 4.4.
\begin{lemma} For $j=1,2,3$, let $V^{n,j}$, $W^{n,j}$ be defined by 
\eqref{eq37}-\eqref{eq310} and let $\lambda = k/h$. Then, there exist constants 
$C_{\lambda}$ depending polynomially on $\lambda$, such that
\begin{align}
& \|H^{n} - V^{n,j}\|_{\infty} + \|U^{n} - W^{n,j}\|_{\infty} \leq C_{\lambda}k,
\label{eq41}\\
& \|H^{n} - V^{n,j}\|_{1,\infty} + \|U^{n} - W^{n,j}\|_{1,\infty} \leq C_{\lambda}.
\label{eq42}
\end{align}
\label{L41}
\end{lemma}
\begin{proof} From \eqref{eq37}, \eqref{eq38}, and \eqref{eq39}, \eqref{eq310},
we have, for $j=1,2,3$,
\begin{align*}
& H^{n} - V^{n,j} = ka_{j}P\Phi_{x}^{n,j-1}
= ka_{j}P(W^{n,j-1} + V^{n,j-1}W^{n,j-1})_{x}, \\
& U^{n} - W^{n,j} = ka_{j}P_{0}F^{n,j-1} = ka_{j}P_{0}(V_{x}^{n,j-1} + W^{n,j-1}W_{x}^{n,j-1}),
\end{align*}
and so, by (2.3b)
\begin{align*}
& \|H^{n} - V^{n,j}\|_{\infty}
\leq Ck(\|W_{x}^{n,j-1}\|_{\infty} + \|V_{x}^{n,j-1}\|_{\infty}\|W^{n,j-1}\|_{\infty}
+ \|V^{n,j-1}\|_{\infty} \|W_{x}^{n,j-1}\|_{\infty}), \\
& \|U^{n} - W^{n,j}\|_{\infty}
\leq Ck(\|V_{x}^{n,j-1}\|_{\infty} + \|W^{n,j-1}\|_{\infty} \|W_{x}^{n,j-1}\|_{\infty}),
\end{align*}
for $j=1,2,3$. From these relations, using e.g. \eqref{eq211} and the inverse properties of $S_{h}$, $S_{h,0}$, we may derive recursively \eqref{eq41} and \eqref{eq42}.
\end{proof}
\begin{lemma} Let $\ve^{n} \in S_{h}$, $e^{n} \in S_{h,0}$ and suppose that $\rho^{n,j}$, $r^{n,j}$
are functions defined for $j=1,2,3$ by
\begin{align}
& \rho^{n,j} = (1 + H^{n})P_{0}r_{x}^{n,j-1} + U^{n}P\rho_{x}^{n,j-1},
\label{eq43}\\
& r^{n,j} = P\rho_{x}^{n,j-1} + U^{n} P_{0}r_{x}^{n,j-1},
\label{eq44}
\end{align}
with
\begin{align}
& \rho^{n,0}=\rho^{n} = (1 + H^{n})e^{n} + U^{n}\ve^{n},
\label{eq45}\\
& r^{n,0} = r^{n} = \ve^{n} + U^{n}e^{n}.
\label{eq46}
\end{align}
Then, there exists a constant $C$ such that
\begin{align}
& \|\rho^{n}\| + \|r^{n}\| \leq C (\|\ve^{n}\| + \|e^{n}\|),
\label{eq47} \\
& \|\rho_{x}^{n,j}\| + \|r_{x}^{n,j}\| \leq \tfrac{C}{h^{j+1}}(\|\ve^{n}\| + \|e^{n}\|),
\quad j=1,2,3.
\label{eq48}
\end{align}
If, moreover, $\|\ve^{n}\|_{1,\infty} + \|e^{n}\|_{1,\infty} \leq \wt{C}$ for some constant $\wt{C}$,  then, for $j=0,1,2,3$,
\begin{equation}
\|\rho_{x}^{n,j}\|_{\infty} + \|r_{x}^{n,j}\|_{\infty} \leq \tfrac{C}{h^{j}}.
\label{eq49}
\end{equation}
\label{L42}
\end{lemma}
\begin{proof} The inequality \eqref{eq47} follows from \eqref{eq45} and \eqref{eq46} and (2.3b).
To prove \eqref{eq48} note that
\[
\|\rho_{x}^{n}\| + \|r_{x}^{n}\| \leq \|H_{x}^{n}e^{n}\| + \|(1+H^{n})e_{x}^{n}\|
+ \|\ve_{x}^{n}\| + \|U_{x}^{n}e^{n}\| + \|U^{n}e_{x}^{n}\|.
\]
Similarly,
\[
\|\rho_{x}^{n,j}\| + \|r_{x}^{n,j}\| \leq \tfrac{C}{h}(\|r_{x}^{n,j-1}\| + \|\rho_{x}^{n,j-1}\|),
\]
for $j=1,2,3$, and \eqref{eq48} follows by recursion. Since
\[
\|\rho_{x}^{n}\|_{\infty} + \|r_{x}^{n}\|_{\infty} \leq \|H_{x}^{n}e^{n}\|_{\infty}
+ \|(1+H^{n})e_{x}^{n}\|_{\infty} + \|\ve_{x}^{n}\|_{\infty} + \|U_{x}^{n}e^{n}\|_{\infty}
+ \|U^{n}e_{x}^{n}\|_{\infty},
\]
the hypothesis of the Lemma and \eqref{eq211}, imply
\[
\|\rho_{x}^{n}\|_{\infty} + \|r_{x}^{n}\|_{\infty} \leq C.
\]
Moreover, since
\[
\|\rho_{x}^{n,j}\|_{\infty} + \|r_{x}^{n,j}\|_{\infty}
\leq \tfrac{C}{h}(\|r_{x}^{n,j-1}\|_{\infty} + \|\rho_{x}^{n,j-1}\|_{\infty})
\]
holds for $j=1,2,3$, a recursive argument yields \eqref{eq49}.
\end{proof}
\begin{lemma}
Given $\ve^{n} \in S_{h}$, $e^{n}\in S_{h,0}$, let $\rho^{n,j}$, $r^{n,j}$, $0\leq j\leq 3$,
be defined as in Lemma \ref{L42}. In addition, let $\rho^{n,-1}(x) = \int_{0}^{x}\ve^{n}$,
$r^{n,-1}(x) = \int_{0}^{x}e^{n}$, for $0\leq x\leq 1$. Then, for $0\leq i < j \leq 3$,
we have 
\begin{equation}
(P\rho_{x}^{n,i}, \rho_{x}^{n,j}) + \bigl((1 + H^{n})P_{0}r_{x}^{n,i},r_{x}^{n,j}\bigr) =
-(\rho_{x}^{n,i+1},P\rho_{x}^{n,j-1}) - \bigl((1 + H^{n})r_{x}^{n,i+1}, P_{0}r_{x}^{n,j-1}\bigr)
+ \gamma_{i}^{n,j-1},
\label{eq410}
\end{equation}
where
\begin{equation}
\gamma_{i}^{n,j-1} = (U_{x}^{n}P\rho_{x}^{n,i},P\rho_{x}^{n,j-1})
+ \bigl([(1 + H^{n})U_{x}^{n} - H_{x}^{n}U^{n}]P_{0}r_{x}^{n,i},P_{0}r_{x}^{n,j-1}\bigr).
\label{eq411}
\end{equation}
In the particular cases $j=i+1$, $i=-1,0,1,2$, \eqref{eq410} may be simplified to:
\begin{equation}
(P\rho_{x}^{n,i},\rho_{x}^{n,i+1}) + \bigl((1 + H^{n})P_{0}r_{x}^{n,i},r_{x}^{n,i+1}\bigr)
= \tfrac{1}{2}\gamma_{i}^{n,i}.
\label{eq412}
\end{equation}
In all cases $-1\leq i<j\leq 3$, we have the estimates
\begin{equation}
\abs{\gamma_{i}^{n,j-1}} \leq \tfrac{C}{h^{i+j+1}}(\|\ve^{n}\|^{2} + \|e^{n}\|^{2}).
\label{eq413}
\end{equation}
\label{L43}
\end{lemma}
\begin{proof}
In all cases $-1\leq i<j\leq 3$, since $\rho^{n,j}$, $P_{0}r_{x}^{n,i}$ vanish at $x=0,1$,
integrating by parts yields
\begin{equation}
(P\rho_{x}^{n,i}, \rho_{x}^{n,j}) + \bigl((1 + H^{n})P_{0}r_{x}^{n,i},r_{x}^{n,j}\bigr) =
-\bigl((P\rho_{x}^{n,i})_{x},\rho^{n,i}\bigr)
- \bigl([(1+H^{n})P_{0}r_{x}^{n,i}]_{x}, r^{n,j}\bigr) \equiv A^{i,j}.
\label{eq414}
\end{equation}
We first examine the cases with $0\leq i<j\leq 3$. We have, using the definitions of the
$\rho^{n,\alpha}$, $r^{n,\alpha}$ and some computation, that
\begin{align*}
A^{i,j} & = -\bigl((P\rho_{x}^{n,i})_{x},(1+H^{n})P_{0}r_{x}^{n,j-1} + U^{n}P\rho_{x}^{n,j-1}\bigr)
- \bigl([(1 + H^{n})P_{0}r_{x}^{n,i}]_{x},P\rho_{x}^{n,j-1} + U^{n}P_{0}r_{x}^{n,j-1}\bigr)\\
& = -\bigr([(1+H^{n})P_{0}r_{x}^{n,i} + U^{n}P\rho_{x}^{n,i}]_{x}-U_{x}^{n}P\rho_{x}^{n,i},P\rho_{x}^{n,j-1}\bigr)  \\
&\,\,\,\,\,\,\, - \bigl((1+H^{n})(P\rho_{x}^{n,i}+U^{n}P_{0}r_{x}^{n,i})_{x} - (1+H^{n})U_{x}^{n}P_{0}r_{x}^{n,i} + H_{x}^{n}U^{n}P_{0}r_{x}^{n,i},P_{0}r_{x}^{n,j-1}\bigr)\\
& = - (\rho_{x}^{n,i+1},P\rho_{x}^{n,j-1}) - \bigl((1+H^{n})r_{x}^{n,i+1},P_{0}r_{x}^{n,j-1}\bigr)
+ \gamma_{i}^{n,j-1},
\end{align*}
where the $\gamma_{i}^{n,j-1}$ are defined by \eqref{eq411}. The last equality above and
\eqref{eq414} give \eqref{eq410}. The remaining case $i=-1$, $j=0$ is a special case of
\eqref{eq412}; the latter follows from \eqref{eq414}, \eqref{eq411}, similar computations as
above, and the identity $(\alpha v,(\beta v)_{x}) = ((\alpha\beta_{x} - \alpha_{x}\beta)v,v)/2$ valid
for $\alpha$, $\beta \in H^{1}$, $v \in \acr{H}^{1}$. \par
The estimate \eqref{eq413} for $j=1,2,3$, $0\leq i<j$, follows from \eqref{eq411}, and \eqref{eq48}, (2.3b). If $i=-1$ the proof is similar and takes account of the facts that $\rho_{x}^{n,-1}=\ve^{n}$, $r_{x}^{n,-1} = e^{n}$.
\end{proof}
The main error estimate of the paper, which incorporates the crucial stability step applied to an
error energy inequality, follows. We remark that the proof does not use the hypothesis \eqref{eq212}
except in its last step, where the local error estimates of Proposition \ref{P31} 
(recall that the latter
rely on the validity of \eqref{eq212}) are brought to bear.
\begin{theorem} Let $(\eta,u)$ be the solution of (SW) in $[0,1]\times[0,T]$ with
$1 + \eta \geq \alpha >0$, for some constant $\alpha$, and let  
$(H_{h}^{n},U_{h}^{n})$ be its fully discrete approximation defined by \eqref{eq33}-\eqref{eq36}. If $\lambda=k/h$ and $h$ is sufficiently
small, then there exists a constant $\lambda_{0}$ depending on $\alpha$, and a constant $C$
independent of $k$ and $h$, such that for $\lambda \leq \lambda_{0}$,
\begin{equation}
\max_{0\leq n\leq M}(\|\eta(t^{n}) - H_{h}^{n}\| + \|u(t^{n}) - U_{h}^{n}\|)
\leq C(h^{r-1} + k^{4}).
\label{eq415}
\end{equation}
\label{T41}
\end{theorem}
\begin{proof} It suffices to show that
\begin{equation}
\max_{0\leq n\leq  M}(\|H^{n} - H_{h}^{n}\| + \|U^{n} - U_{h}^{n}\|) \leq C(h^{r-1} + k^{4}).
\label{eq416}
\end{equation}
To make the exposition easier to follow, we break up the proof in five parts. \\ \noindent
(i)\, {\em{Notation and the basic error equations}} \\ \noindent
Let
\[
\ve^{n,j}=V^{n,j} - H_{h}^{n,j}, \quad e^{n,j} = W^{n,j} - U_{h}^{n,j}, \quad j=0,1,2,3,
\]
with $\ve^{n,0} = \ve^{n}=H^{n}-H_{h}^{n}$, $e^{n,0} = e^{n} = U^{n}-U_{h}^{n}$. Then, from \eqref{eq33},
\eqref{eq37}, \eqref{eq38} it follows for $j=1,2,3$ that
\begin{equation}
\begin{aligned}
\ve^{n,j} & = \ve^{n} - a_{j}kP(\Phi^{n,j-1} - \Phi_{h}^{n,j-1})_{x}, \\
e^{n,j} & = e^{n} - a_{j}kP_{0}(F^{n,j-1} - F_{h}^{n,j-1}),
\end{aligned}
\label{eq417}
\end{equation}
and from \eqref{eq34}, \eqref{eq314}, \eqref{eq315} that
\begin{equation}
\begin{aligned}
\ve^{n+1} & = \ve^{n} - kP\Bigl[\sum_{j=1}^{4}b_{j}(\Phi^{n,j-1} - \Phi_{h}^{n,j-1})_{x}\Bigr]
+ \delta_{1}^{n}, \\
e^{n+1} & = e^{n} - kP_{0}\Bigl[\sum_{j=1}^{4}b_{j}(F^{n,j-1} - F_{h}^{n,j-1})\Bigr] + \delta_{2}^{n}.
\end{aligned}
\label{eq418}
\end{equation}
Also, from \eqref{eq39}, \eqref{eq310}, \eqref{eq35}, we have for $j=0,1,2,3$
\begin{align*}
\Phi^{n,j} - \Phi_{h}^{n,j} & = W^{n,j} + V^{n,j}W^{n,j} - U_{h}^{n,j} - H_{h}^{n,j}U_{h}^{n,j}\\
& = e^{n,j} + V^{n,j}W^{n,j} - (V^{n,j} - \ve^{n,j})(W^{n,j} - e^{n,j}),
\end{align*}
and
\begin{align*}
F^{n,j} - F_{h}^{n,j} & = V_{x}^{n,j} + W^{n,j}W_{x}^{n,j} - H_{hx}^{n,j} - U_{h}^{n,j}U_{hx}^{n,j}\\
& = \ve_{x}^{n,j} + W^{n,j}W_{x}^{n,j} - (W^{n,j} - e^{n,j})(W_{x}^{n,j} - e_{x}^{n,j}),
\end{align*}
so that
\begin{equation}
\begin{aligned}
\Phi^{n,j} - \Phi_{h}^{n,j} & = (1 + V^{n,j})e^{n,j} + W^{n,j}\ve^{n,j} - \ve^{n,j}e^{n,j},\\
F^{n,j} - F_{h}^{n,j} & = \ve_{x}^{n,j} + (W^{n,j}e^{n,j})_{x} -\tfrac{1}{2}((e^{n,j})^{2})_{x}.
\end{aligned}
\label{eq419}
\end{equation}
Since now
\begin{align*}
(1 + V^{n,j})e^{n,j} & = \bigl[1 + H^{n} - (H^{n} - V^{n,j})\bigr]e^{n,j} =
(1 + H^{n})e^{n,j} - (H^{n} - V^{n,j})e^{n,j},\\
W^{n,j}\ve^{n,j} & = \bigl[U^{n} - (U^{n} - W^{n,j})\bigr]\ve^{n,j} =
U^{n}\ve^{n,j} - (U^{n} - W^{n,j})\ve^{n,j},\\
W^{n,j}e^{n,j} & = \bigl[U^{n} - (U^{n} - W^{n,j})\bigr]e^{n,j} =
U^{n}e^{n,j} - (U^{n} - W^{n,j})e^{n,j},
\end{align*}
it follows from the equations \eqref{eq419} that for $j=0,1,2,3$
\begin{equation}
\begin{aligned}
\Phi^{n,j} - \Phi_{h}^{n,j} & = (1 + H^{n})e^{n,j}  + U^{n}\ve^{n,j} + v^{n,j}, \\
F^{n,j} - F_{h}^{n,j} & = \ve_{x}^{n,j} + (U^{n}e^{n,j})_{x}  + w_{x}^{n,j},
\end{aligned}
\label{eq420}
\end{equation}
where
\begin{equation}
\begin{aligned}
v^{n,j} & = - (H^{n} - V^{n,j})e^{n,j}
 -(U^{n} - W^{n,j})\ve^{n,j} - \ve^{n,j}e^{n,j},\\
w^{n,j} & = - (U^{n} - W^{n,j})e^{n,j} - \tfrac{1}{2}(e^{n,j})^{2}.
\end{aligned}
\label{eq421}
\end{equation}
We note for future reference that it follows from \eqref{eq421}, the inequalities 
\eqref{eq41}, and \eqref{eq42} and the inverse inequalities on the spaces 
$S_{h}$, $S_{h,0}$ that
\begin{align}
&\|v_{x}^{n,j}\| + \|w_{x}^{n,j}\| \leq (C_{\lambda} + \|\ve^{n,j}\|_{1,\infty}
+ \|e^{n,j}\|_{1,\infty})(\|\ve^{n,j}\| + \|e^{n,j}\|),
\label{eq422} \\
& \|v^{n,j}_{x}\|_{\infty} + \|w_{x}^{n,j}\|_{\infty} \leq
C_{\lambda}(\|\ve^{n,j}\|_{\infty} + \|e^{n,j}\|_{\infty})
+ 3\|\ve^{n,j}\|_{1,\infty}\|e^{n,j}\|_{1,\infty},
\label{eq423}
\end{align}
where, as usual, $C_{\lambda}$ denotes a constant depending polynomially on $\lambda$. \\ \noindent
(ii)\, {\em{Expansions of $\Phi^{n,j} - \Phi_{h}^{n,j}$, $F^{n,j} - F_{h}^{n,j}$ in powers of $k$}} \\
In this part of the proof we derive suitable representations of the differences
$\Phi^{n,j} - \Phi_{h}^{n,j}$, $F^{n,j} - F_{h}^{n,j}$ that will be used in the energy identities. \par
If $j=0$, we have
\begin{equation}
\begin{aligned}
\Phi^{n} - \Phi_{h}^{n} & = \rho^{n} + \rho_{1}^{n},\\
F^{n} - F_{h}^{n} & = r_{x}^{n} + r_{1x}^{n},
\end{aligned}
\label{eq424}
\end{equation}
where the $\rho^{n}$, $r^{n}$ were defined in Lemma \ref{L42} and satisfy the inequalities \eqref{eq47} and
\eqref{eq48}, and
\begin{equation}
\begin{aligned}
\rho_{1}^{n} & = v^{n,0} = - \ve^{n}e^{n}, \\
r_{1}^{n} & = w^{n,0} =  -\tfrac{1}{2}(e^{n})^{2},
\end{aligned}
\label{eq425}
\end{equation}
From \eqref{eq417} and \eqref{eq424} it follows that
\begin{equation}
\begin{aligned}
\ve^{n,1} & = \ve^{n} - ka_{1}P\rho_{x}^{n} - ka_{1}P\rho_{1x}^{n},\\
e^{n,1} & = e^{n} - ka_{1}P_{0}r_{x}^{n} - ka_{1}P_{0}r_{1x}^{n},
\end{aligned}
\label{eq426}
\end{equation}
and, using the notation of Lemma \ref{L42}, \eqref{eq43}, and \eqref{eq44}, that
\begin{align*}
& (1 + H^{n})e^{n,1} + U^{n}\ve^{n,1} = \rho^{n} - ka_{1}\rho^{n,1}
- ka_{1}[(1+H^{n})P_{0}r_{1x}^{n} + U^{n}P\rho_{1x}^{n}],\\
& \ve^{n,1} + U^{n}e^{n,1} = r^{n} - ka_{1}r^{n,1}
- ka_{1}(P\rho_{1x}^{n}+U^{n}P_{0}r_{1x}^{n}).
\end{align*}
Consequently, from the equations \eqref{eq420} we obtain
\begin{equation}
\begin{aligned}
\Phi^{n,1} - \Phi_{h}^{n,1} & = \rho^{n} - ka_{1}\rho^{n,1} + \rho_{1}^{n,1},\\
F^{n,1} - F_{h}^{n,1} & = r_{x}^{n} - ka_{1}r_{x}^{n,1} + r_{1x}^{n,1},
\end{aligned}
\label{eq427}
\end{equation}
where
\begin{equation}
\begin{aligned}
\rho_{1}^{n,1} & = - ka_{1}[(1+H^{n})P_{0}r_{1x}^{n} + U^{n}P\rho_{1x}^{n}]  + v^{n,1}, \\
r_{1}^{n,1} & = - ka_{1}(P\rho_{1x}^{n}+U^{n}P_{0}r_{1x}^{n}) + w^{n,1}.
\end{aligned}
\label{eq428}
\end{equation}
Hence, from \eqref{eq417} and the equations \eqref{eq427} we get
\begin{equation}
\begin{aligned}
\ve^{n,2} & =  \ve^{n} - ka_{2}P\rho_{x}^{n} + k^{2}a_{1}a_{2}P\rho_{x}^{n,1}
- ka_{2}P\rho_{1x}^{n,1},\\
e^{n,2} & = e^{n} - ka_{2}P_{0}r_{x}^{n} + k^{2}a_{1}a_{2}P_{0}r_{x}^{n,1}
- ka_{2}P_{0}r_{1x}^{n,1},
\end{aligned}
\label{eq429}
\end{equation}
and, using the notation introduced in \eqref{eq43} and \eqref{eq44},
\begin{align*}
& (1 + H^{n})e^{n,2} + U^{n}\ve^{n,2} = \rho^{n} -ka_{2}\rho^{n,1} + k^{2}a_{1}a_{2}\rho^{n,2}
- ka_{2}[(1 + H^{n})P_{0}r_{1x}^{n,1} + U^{n}P\rho_{1x}^{n,1}], \\
& \ve^{n,2} + U^{n}e^{n,2} = r^{n} - ka_{2}r^{n,1} + k^{2}a_{1}a_{2}r^{n,2}
- ka_{2}(P\rho_{1x}^{n,1} + U^{n}P_{0}r_{1x}^{n,1}).
\end{align*}
So, from the equations \eqref{eq420} we see that
\begin{equation}
\begin{aligned}
\Phi^{n,2} - \Phi_{h}^{n,2} & = \rho^{n} - ka_{2}\rho^{n,1} + k^{2}a_{1}a_{2}\rho^{n,2}
 + \rho_{1}^{n,2}, \\
F^{n,2} - F_{h}^{n,2} & = r_{x}^{n} - ka_{2}r_{x}^{n,1} + k^{2}a_{1}a_{2}r_{x}^{n,2}
 + r_{1x}^{n,2},
\end{aligned}
\label{eq430}
\end{equation}
where
\begin{equation}
\begin{aligned}
\rho_{1}^{n,2} & = - ka_{2}[(1 + H^{n})P_{0}r_{1x}^{n,1} + U^{n}P\rho_{1x}^{n,1}] + v^{n,2}, \\
r_{1}^{n,2} & = - ka_{2}(P\rho_{1x}^{n,1} + U^{n}P_{0}r_{1x}^{n,1}) + w^{n,2}.
\end{aligned}
\label{eq431}
\end{equation}
Hence, from the equations \eqref{eq417}, taking into account that $a_{3}=1$, we obtain
\begin{equation}
\begin{aligned}
\ve^{n,3} & = \ve^{n} - kP\rho_{x}^{n} + k^{2}a_{2}P\rho_{x}^{n,1}
 - k^{3}a_{1}a_{2}P\rho_{x}^{n,2} - kP\rho_{1x}^{n,2},\\
e^{n,3} & = e^{n} - kP_{0}r_{x}^{n} + k^{2}a_{2}P_{0}r_{x}^{n,1}
- k^{3}a_{1}a_{2}P_{0}r_{x}^{n,2} -  kP_{0}r_{1x}^{n,2},
\end{aligned}
\label{eq432}
\end{equation}
and, according to \eqref{eq43} and \eqref{eq44},
\begin{align*}
& (1 + H^{n})e^{n,3} + U^{n}\ve^{n,3} = \rho^{n} - k\rho^{n,1} + k^{2}a_{2}\rho^{n,2}
- k^{3}a_{1}a_{2}\rho^{n,3} - k[(1+H^{n})P_{0}r_{1x}^{n,2} + U^{n}P\rho_{1x}^{n,2}], \\
& \ve^{n,3} + U^{n}e^{n,3} = r^{n} - kr^{n,1} + k^{2}a_{2}r^{n,2}
- k^{3}a_{1}a_{2}r^{n,3} - k(P\rho_{1x}^{n,2} + U^{n}P_{0}r_{1x}^{n,2}).
\end{align*}
Consequently,
\begin{equation}
\begin{aligned}
\Phi^{n,3} - \Phi_{h}^{n,3} & = \rho^{n} - k\rho^{n,1} + k^{2}a_{2}\rho^{n,2}
- k^{3}a_{1}a_{2}\rho^{n,3} + \rho_{1}^{n,3},\\
F^{n,3} - F_{h}^{n,3} & = r_{x}^{n} - kr_{x}^{n,1} + k^{2}a_{2}r_{x}^{n,2}
- k^{3}a_{1}a_{2}r_{x}^{n,3} + r_{1x}^{n,3},
\end{aligned}
\label{eq433}
\end{equation}
where
\begin{equation}
\begin{aligned}
\rho_{1}^{n,3} & = - k[(1+H^{n})P_{0}r_{1x}^{n,2} + U^{n}P\rho_{1x}^{n,2}] + v^{n,3},\\
r_{1}^{n,3} & = - k(P\rho_{1x}^{n,2} + U^{n}P_{0}r_{1x}^{n,2}) + w^{n,3}.
\end{aligned}
\label{eq434}
\end{equation}
(iii)\, {\em{`Inductive' hypothesis and consequent estimates}} \\
We let now $n^{*}$ be the maximal integer for which
\[
\|\ve^{n}\|_{1,\infty} + \|e^{n}\|_{1,\infty} \leq 1, \quad 0\leq n\leq n^{*}.
\tag{H}
\label{eqh}
\]
Then, for $0\leq n\leq n^{*}$, from \eqref{eq425} it follows that
\begin{equation}
\begin{aligned}
& \|\rho_{1x}^{n}\| + \|r_{1x}^{n}\| \leq C(\|\ve^{n}\| + \|e^{n}\|),\\
& \|\rho_{1x}^{n}\|_{\infty} + \|r_{1x}^{n}\|_{\infty} \leq C.
\end{aligned}
\label{eq435}
\end{equation}
In addition, from \eqref{eq426}, \eqref{eqh}, \eqref{eq425}, 
\eqref{eq48} for $j=0,1$, \eqref{eq49} for $j=0$, the
second inequality of \eqref{eq435}, the inverse properties of $S_{h}$, $S_{h,0}$, 
and (2.3b), it follows that
\begin{align}
& \|\ve^{n,1}\| + \|e^{n,1}\| \leq C_{\lambda}(\|\ve^{n}\| + \|e^{n}\|),
\label{eq436} \\
& \|\ve^{n,1}\|_{1,\infty} + \|e^{n,1}\|_{1,\infty} \leq C_{\lambda}.
\label{eq437}
\end{align}
Also, from \eqref{eq422}, \eqref{eq423}, \eqref{eq436}, and \eqref{eq437}, we have
\begin{align}
& \|v_{x}^{n,1}\| + \|w_{x}^{n,1}\| \leq C_{\lambda}(\|\ve^{n}\| + \|e^{n}\|),
\label{eq438} \\
& \|v_{x}^{n,1}\|_{\infty} + \|w_{x}^{n,1}\|_{\infty} \leq C_{\lambda}.
\label{eq439}
\end{align}
Now, for $j=1,2,3$, in view of \eqref{eq428}, \eqref{eq431}, \eqref{eq434}, and the inverse properties of $S_{h}$ and $S_{h,0}$, it follows that
\begin{align*}
& \|\rho_{1x}^{n,j}\| + \|r_{1x}^{n,j}\|
\leq C_{\lambda}(\|\rho_{1x}^{n,j-1}\| + \|r_{1x}^{n,j-1}\|)
+ (\|v_{x}^{n,j}\| + \|w_{x}^{n,j}\|),\\
& \|\rho_{1x}^{n,j}\|_{\infty} + \|r_{1x}^{n,j}\|_{\infty}
\leq C_{\lambda}(\|\rho_{1x}^{n,j-1}\|_{\infty} + \|r_{1x}^{n,j-1}\|_{\infty})
+ (\|v_{x}^{n,j}\|_{\infty} + \|w_{x}^{n,j}\|_{\infty}).
\end{align*}
In addition, for $j=2,3$, \eqref{eq47}, \eqref{eq48}, \eqref{eq429}, \eqref{eq432}, \eqref{eq49}, and the inverse properties of $S_{h}$, $S_{h,0}$ give
\begin{align*}
& \|\ve^{n,j}\| + \|e^{n,j}\| \leq C_{\lambda}(\|\ve^{n}\| + \|e^{n}\|)
+ Ck(\|\rho_{1x}^{n,j-1}\| + \|r_{1x}^{n,j-1}\|), \\
& \|\ve^{n,j}\|_{1,\infty} + \|e^{n,j}\|_{1,\infty}
\leq C_{\lambda}(1 + \|\rho_{1x}^{n,j-1}\|_{\infty} + \|r_{1x}^{n,j-1}\|_{\infty}),
\quad 0\leq n\leq n^{*}.
\end{align*}
Therefore, for $0\leq n\leq n^{*}$ and $j=0,1,2,3$, in view of \eqref{eq435}, \eqref{eq438}, \eqref{eq439}, \eqref{eq422}, and arguing recursively, we finally obtain
\begin{align}
& \|\rho_{1x}^{n,j}\| + \|r_{1x}^{n,j}\| \leq C_{\lambda}(\|\ve^{n}\| + \|e^{n}\|),
\label{eq440} \\
& \|\rho_{1x}^{n,j}\|_{\infty} + \|r_{1x}^{n,j}\|_{\infty} \leq C_{\lambda}.
\label{eq441}
\end{align}
(iv)\, {\em{Basic energy identity and estimation of the terms in its right-hand side}} \\
From \eqref{eq418}, \eqref{eq424}, \eqref{eq427}, \eqref{eq430}, \eqref{eq433}, and the definitions of the constants $a_{j}$, $b_{j}$ of the RK scheme we have
\begin{equation}
\begin{aligned}
\ve^{n+1} & = f^{n} + f_{1}^{n} + \delta_{1}^{n}, \\
e^{n+1} & = g^{n} + g_{1}^{n} + \delta_{2}^{n},
\end{aligned}
\label{eq442}
\end{equation}
where
\begin{align*}
f^{n} & = \ve^{n} - kP\rho_{x}^{n} + \tfrac{k^{2}}{2}P\rho_{x}^{n,1}
- \tfrac{k^{3}}{6}P\rho_{x}^{n,2} + \tfrac{k^{4}}{24}P\rho_{x}^{n,3},\\
g^{n} & = e^{n} - kP_{0}r_{x}^{n} + \tfrac{k^{2}}{2}P_{0}r_{x}^{n,1}
-\tfrac{k^{3}}{6}P_{0}r_{x}^{n,2} + \tfrac{k^{4}}{24}P_{0}r_{x}^{n,3},\\
f_{1}^{n} & = - \tfrac{k}{6}(P\rho_{1x}^{n} + 2P\rho_{1x}^{n,1} + 2P\rho_{1x}^{n,2}
+ P\rho_{1x}^{n,3}), \\
g_{1}^{n} & = -\tfrac{k}{6}(P_{0}r_{1x}^{n} + 2P_{0}r_{1x}^{n,1} + 2P_{0}r_{1x}^{n,2}
+ P_{0}r_{1x}^{n,3}).
\end{align*}
From these relations and \eqref{eq47}, \eqref{eq48} it follows that
\begin{equation}
\|f^{n}\| + \|g^{n}\| \leq C_{\lambda} (\|\ve^{n}\| + \|e^{n}\|),
\label{eq443}
\end{equation}
and, moreover, for $0\leq n\leq n^{*}$, from \eqref{eq440}
\begin{equation}
\|f_{1}^{n}\| + \|g_{1}^{n}\|\leq C_{\lambda}k(\|\ve^{n}\| + \|e^{n}\|).
\label{eq444}
\end{equation}
Now, by the definitions of $f^{n}$, $g^{n}$, we may obtain the basic energy identity of our scheme:
\begin{equation}
\|f^{n}\|^{2} + \bigl((1+H^{n})g^{n},g^{n}\bigr) =
\|\ve^{n}\|^{2} + \bigl((1+H^{n})e^{n},e^{n}) + \sum_{i=1}^{8}k^{i}\beta_{i}^{n}.
\label{eq445}
\end{equation}
We will now identify and estimate the quantities $\beta_{i}^{n}$, $1\leq i\leq 8$, 
in the right-hand side of the above. For $\beta_{1}^{n}$ we have
\[
\beta_{1}^{n} = -2(\ve^{n},P\rho_{x}^{n}) - 2\bigl((1+H^{n})e^{n},P_{0}r_{x}^{n}\bigr).
\]
Since, by \eqref{eq412},
\[
(\ve^{n},\rho_{x}^{n}) + \bigl((1 + H^{n})e^{n},r_{x}^{n}\bigr) = \tfrac{1}{2}\gamma_{-1}^{n,-1},
\]
it follows that
\[
\beta_{1}^{n} = -2 \bigl((1+ H^{n})e^{n},P_{0}r_{x}^{n}-r_{x}^{n}\bigr) - \gamma_{-1}^{n,-1}.
\]
From this relation, Lemma \ref{L21}(ii), \eqref{eq48}, and \eqref{eq413}, we see that
\begin{equation}
\abs{\beta_{1}^{n}} \leq C(\|\ve^{n}\|^{2} + \|e^{n}\|^{2}).
\label{eq446}
\end{equation}
The quantity $\beta_{2}^{n}$ is given by
\[
\beta_{2}^{n} = (\ve^{n},P\rho_{x}^{n,1}) + \bigl((1+H^{n})e^{n},P_{0}r_{x}^{n,1}\bigr)
+ \|P\rho_{x}^{n}\|^{2} + \bigl((1+H^{n})P_{0}r_{x}^{n},P_{0}r_{x}^{n}\bigr).
\]
Since, by \eqref{eq410}
\[
(\ve^{n},\rho_{x}^{n,1}) + \bigl((1+H^{n})e^{n},r_{x}^{n,1})  =
- (\rho_{x}^{n},P\rho_{x}^{n}) - \bigl((1 + H^{n})r_{x}^{n},P_{0}r_{x}^{n}\bigr)
+ \gamma_{-1}^{n,0},
\]
we see that
\[
\beta_{2}^{n} = \bigl((1 + H^{n})e^{n},P_{0}r_{x}^{n,1} - r_{x}^{n,1}\bigr)
+ \bigl((1 + H^{n})P_{0}r_{x}^{n}, P_{0}r_{x}^{n} - r_{x}^{n}\bigr) + \gamma_{-1}^{n,0}.
\]
Hence, by Lemma \ref{L21}(ii), \eqref{eq48}, and \eqref{eq413} it follows that
\begin{equation}
\abs{\beta_{2}^{n}} \leq \tfrac{C}{h}(\|\ve^{n}\|^{2} + \|e^{n}\|^{2}).
\label{eq447}
\end{equation}
For $\beta_{3}^{n}$ we find
\[
\beta_{3}^{n} = -\tfrac{1}{3}\bigl((\ve^{n},P\rho_{x}^{n,2}) + ((1+H^{n})e^{n},P_{0}r_{x}^{n,2})\bigr)
- \bigl((P\rho_{x}^{n},P\rho_{x}^{n,1}) + ((1+H^{n})P_{0}r_{x}^{n},P_{0}r_{x}^{n,1})\bigr),
\]
i.e.
\begin{align*}
\beta_{3}^{n} & = -\tfrac{1}{3}\bigl[(\ve^{n}, P\rho_{x}^{n,2})
+ \bigl((1 + H^{n})e^{n},P_{0}r_{x}^{n,2}\bigr) + (P\rho_{x}^{n},P\rho_{x}^{n,1})
+ \bigl((1 + H^{n})P_{0}r_{x}^{n},P_{0}r_{x}^{n,1}\bigr)\bigr] \\
& \,\,\,\,\,\,\, - \tfrac{2}{3}\bigl[(P\rho_{x}^{n},P\rho_{x}^{n,1})
 + \bigl((1 + H^{n})P_{0}r_{x}^{n},P_{0}r_{x}^{n,1}\bigr)\bigr].
\end{align*}
Using \eqref{eq410}, \eqref{eq412} we see that
\begin{align*}
& (\ve^{n},P\rho_{x}^{n,2}) + \bigl((1 + H^{n})e^{n},r_{x}^{n,2}\bigr) =
- (\rho_{x}^{n},P\rho_{x}^{n,1}) - \bigl((1 + H^{n})r_{x}^{n},P_{0}r_{x}^{n,1}\bigr)
+ \gamma_{-1}^{n,1}, \\
& (P\rho_{x}^{n},\rho_{x}^{n,1}) + \bigl((1 + H^{n})P_{0}r_{x}^{n},r_{x}^{n,1}\bigr) =
\tfrac{1}{2}\gamma_{0}^{n,0},
\end{align*}
whence
\begin{align*}
\beta_{3}^{n} & = - \tfrac{1}{3}\bigl[\bigl((1 + H^{n})e^{n},P_{0}r_{x}^{n,2}-r_{x}^{n,2}\bigr)
+ \bigl((1 + H^{n})P_{0}r_{x}^{n,1}, P_{0}r_{x}^{n} - r_{x}^{n}\bigr) + \gamma_{-1}^{n,1}\bigr]\\
& \,\,\,\,\,\,\,
- \tfrac{2}{3}\bigl[\bigl((1 + H^{n})P_{0}r_{x}^{n},P_{0}r_{x}^{n,1}-r_{x}^{n,1}\bigr)
+ \tfrac{1}{2}\gamma_{0}^{n,0}\bigr],
\end{align*}
and, therefore, using again Lemma \ref{L21}(ii), \eqref{eq48}, and \eqref{eq413} 
we may estimate $\beta_{3}^{n}$ as
\begin{equation}
\abs{\beta_{3}^{n}} \leq \tfrac{C}{h^{2}}(\|\ve^{n}\|^{2} + \|e^{n}\|^{2}).
\label{eq448}
\end{equation}
For $\beta_{4}^{n}$ there holds
\begin{align*}
\beta_{4}^{n} & = \tfrac{1}{12}\bigl[(\ve^{n},P\rho_{x}^{n,3}) + \bigl((1+H^{n})e^{n},P_{0}r_{x}^{n,3}\bigr)\bigr]
+ \tfrac{1}{3}\bigl[(P\rho_{x}^{n},P\rho_{x}^{n,2})
+ \bigl((1+H^{n})P_{0}r_{x}^{n},P_{0}r_{x}^{n,2}\bigr)\bigr] \\
& \,\,\,\,\,\,\, + \tfrac{1}{4}\bigl[ \|P\rho_{x}^{n,1}\|^{2}
+ \bigl((1 + H^{n})P_{0}r_{x}^{n,1},P_{0}r_{x}^{n,1}\bigr)\bigr],
\end{align*}
or
\begin{align*}
\beta_{4}^{n} & = \tfrac{1}{12}\bigl[(\ve^{n},P\rho_{x}^{n,3})
+ \bigl((1 + H^{n})e^{n},P_{0}r_{x}^{n,3}\bigr) + (P\rho_{x}^{n},P\rho_{x}^{n,2})
+ \bigl((1+H^{n})P_{0}r_{x}^{n},P_{0}r_{x}^{n,2}\bigr)\bigr] \\
&\,\,\,\,\,\,\, + \tfrac{1}{4}\bigl[ (P\rho_{x}^{n},P\rho_{x}^{n,2})
+\bigl((1 + H^{n})P_{0}r_{x}^{n},P_{0}r_{x}^{n,2}\bigr)
+ \|P\rho_{x}^{n,1}\|^{2} + \bigl((1 + H^{n})P_{0}r_{x}^{n,1},P_{0}r_{x}^{n,1}\bigr)\bigr].
\end{align*}
Since, in view of \eqref{eq410},
\begin{align*}
& (\ve^{n},P\rho_{x}^{n,3}) + \bigl((1 + H^{n})e^{n},r_{x}^{n,3}\bigr)
= - (\rho_{x}^{n},P\rho_{x}^{n,2}) - \bigl((1+H^{n})r_{x}^{n},P_{0}r_{x}^{n,2}\bigr)
+ \gamma_{-1}^{n,2}, \\
& (P\rho_{x}^{n},\rho_{x}^{n,2}) + \bigl((1 + H^{n})P_{0}r_{x}^{n},r_{x}^{n,2}\bigr) =
- (\rho_{x}^{n,1},P\rho_{x}^{n,1}) - \bigl((1+H^{n})r_{x}^{n,1},P_{0}r_{x}^{n,1}\bigr)
+ \gamma_{0}^{n,1},
\end{align*}
it follows from Lemma \ref{L21}(ii), \eqref{eq48}, and \eqref{eq413} that
\begin{align*}
\beta_{4}^{n} & = \tfrac{1}{12}\bigl[\bigl((1+H^{n})e^{n},P_{0}r_{x}^{n,3}-r_{x}^{n,3}\bigr)
 + \bigl((1+H^{n})P_{0}r_{x}^{n,2},P_{0}r_{x}^{n} - r_{x}^{n}\bigr) + \gamma_{-1}^{n,2}\bigr]\\
&\,\,\,\,\,\,\,
+ \tfrac{1}{4}\bigl[\bigl((1+H^{n})P_{0}r_{x}^{n},P_{0}r_{x}^{n,2}-r_{x}^{n,2}\bigr)
+ \bigl((1+H^{n})P_{0}r_{x}^{n,1},P_{0}r_{x}^{n,1}- r_{x}^{n,1}\bigr)
+ \gamma_{0}^{n,1}\bigr],
\end{align*}
and, consequently, that
\begin{equation}
\abs{\beta_{4}^{n}} \leq \tfrac{C}{h^{3}}(\|\ve^{n}\|^{2} + \|e^{n}\|^{2}).
\label{eq449}
\end{equation}
The quantity $\beta_{5}^{n}$ is given by
\[
\beta_{5}^{n}  = -\tfrac{1}{12}\bigl[(P\rho_{x}^{n},P\rho_{x}^{n,3})
+ \bigl((1+H^{n})P_{0}r_{x}^{n},P_{0}r_{x}^{n,3}\bigr)\bigr]
 - \tfrac{1}{6}\bigl[(P\rho_{x}^{n,1},P\rho_{x}^{n,2})
+\bigl((1+H^{n})P_{0}r_{x}^{n,1},P_{0}r_{x}^{n,2}\bigr)\bigr],
\]
or by
\begin{align*}
\beta_{5}^{n} & = -\tfrac{1}{12}\bigl[(P\rho_{x}^{n},P\rho_{x}^{n,3})
+ \bigl((1+H^{n})P_{0}r_{x}^{n},P_{0}r_{x}^{n,3}\bigr) +
(P\rho_{x}^{n,1},P\rho_{x}^{n,2})
+\bigl((1+H^{n})P_{0}r_{x}^{n,1},P_{0}r_{x}^{n,2}\bigr)\bigr]\\
&\,\,\,\,\,\,\, - \tfrac{1}{12}\bigl[(P\rho_{x}^{n,1},P\rho_{x}^{n,2})
+\bigl((1+H^{n})P_{0}r_{x}^{n,1},P_{0}r_{x}^{n,2}\bigr)\bigr].
\end{align*}
However, in view of \eqref{eq410} and \eqref{eq412}, we have
\begin{align*}
& (P\rho_{x}^{n},\rho_{x}^{n,3}) +
\bigl((1+H^{n})P_{0}r_{x}^{n},r_{x}^{n,3}\bigr) = -(\rho_{x}^{n,1},P\rho_{x}^{n,2})
- \bigl((1+H^{n})r_{x}^{n,1},P_{0}r_{x}^{n,2}\bigr) + \gamma_{0}^{n,2},\\
&(P\rho_{x}^{n,1},\rho_{x}^{n,2}) + \bigl((1+H^{n})P_{0}r_{x}^{n,1},r_{x}^{n,2}\bigr) =
\tfrac{1}{2}\gamma_{1}^{n,1},
\end{align*}
whence, from Lemma \ref{L21}(ii), \eqref{eq48}, and \eqref{eq413}
\begin{align*}
\beta_{5}^{n} & = -\tfrac{1}{12}\bigl[\bigl((1+H^{n})P_{0}r_{x}^{n},P_{0}r_{x}^{n,3}-r_{x}^{n,3}\bigr)
+ \bigl((1+H^{n})P_{0}r_{x}^{n,2},P_{0}r_{x}^{n,1}-r_{x}^{n,1}\bigr)
+ \gamma_{0}^{n,2}\bigr] \\
& \,\,\,\,\,\,\, - \tfrac{1}{12}
\bigl[\bigl((1+H^{n})P_{0}r_{x}^{n,1}, P_{0}r_{x}^{n,2}-r_{x}^{n,2}\bigr)
+ \tfrac{1}{2}\gamma_{1}^{n,1}\bigr],
\end{align*}
and so
\begin{equation}
\abs{\beta_{5}^{n}} \leq \tfrac{C}{h^{4}}(\|\ve^{n}\|^{2} + \|e^{n}\|^{2}).
\label{eq450}
\end{equation}
For $\beta_{6}^{n}$ we have
\[
\beta_{6}^{n}  = \tfrac{1}{24}\bigl[(P\rho_{x}^{n,1},P\rho_{x}^{n,3})
+ \bigl((1+H^{n})P_{0}r_{x}^{n,1},P_{0}r_{x}^{n,3}\bigr)\bigr]
+ \tfrac{1}{36}\bigl[\|P\rho_{x}^{n,2}\|^{2}
+\bigl((1+H^{n})P_{0}r_{x}^{n,2},P_{0}r_{x}^{n,2}\bigr)\bigr],
\]
which gives
\begin{align*}
\beta_{6}^{n} & = \tfrac{1}{24}
\bigr[(P\rho_{x}^{n,1},P\rho_{x}^{n,3})
+ \bigl((1+H^{n})P_{0}r_{x}^{n,1},P_{0}r_{x}^{n,3}\bigr) + \|P\rho_{x}^{n,2}\|^{2}
+\bigl((1+H^{n})P_{0}r_{x}^{n,2},P_{0}r_{x}^{n,2}\bigr)\bigr]\\
&\,\,\,\,\,\,\, - \tfrac{1}{72} \bigl[\|P\rho_{x}^{n,2}\|^{2}
+\bigl((1+H^{n})P_{0}r_{x}^{n,2},P_{0}r_{x}^{n,2}\bigr)\bigr].
\end{align*}
Since now
\[
(P\rho_{x}^{n,1},\rho_{x}^{n,3})
+ \bigl((1 + H^{n})P_{0}r_{x}^{n,1},r_{x}^{n,3}\bigr) =
-(\rho_{x}^{n,2},P\rho_{x}^{n,2}
- \bigl((1+H^{n})r_{x}^{n,2},P_{0}r_{x}^{n,2}\bigr) + \gamma_{1}^{n,2},
\]
we write
\begin{equation}
\beta_{6}^{n} = \beta_{6}^{n,1} + \beta_{6}^{n,2},
\label{eq451}
\end{equation}
where
\begin{align*}
\beta_{6}^{n,1} & = \tfrac{1}{24}\bigl[
\bigl((1+H^{n})P_{0}r_{x}^{n,1},P_{0}r_{x}^{n,3} - r_{x}^{n,3}\bigr)
+\bigl((1+H^{n})P_{0}r_{x}^{n,2},P_{0}r_{x}^{n,2}-r_{x}^{n,2}\bigr) + \gamma_{1}^{n,2}\bigr],\\
\beta_{6}^{n,2} & = -\tfrac{1}{72}\bigl[\|P\rho_{x}^{n,2}\|^{2}
+ \bigl((1+H^{n})P_{0}r_{x}^{n,2},P_{0}r_{x}^{n,2}\bigr)\bigr].
\end{align*}
From \eqref{eq48}, \eqref{eq413}, and Lemma \ref{L21}(ii) we see that
\begin{equation}
\abs{\beta_{6}^{n,1}} \leq \tfrac{C}{h^{5}}(\|\ve^{n}\|^{2} + \|e^{n}\|^{2}).
\label{eq452}
\end{equation}
Now, from Lemma \ref{L22} for sufficiently small $h$ we infer that
\begin{equation}
\beta_{6}^{n,2}\leq 
-\tfrac{C_{\alpha}}{72}(\|P\rho_{x}^{n,2}\|^{2} + \|P_{0}r_{x}^{n,2}\|^{2}),
\label{eq453}
\end{equation}
where $C_{\alpha} = \min(1,\alpha/2)$. \\
The quantity $\beta_{7}^{n}$ is given by
\[
\beta_{7}^{n} = -\tfrac{1}{72}\bigl[(P\rho_{x}^{n,2},P\rho_{x}^{n,3})
+ \bigl((1+H^{n})P_{0}r_{x}^{n,2},P_{0}r_{x}^{n,3}\bigr)\bigr],
\]
and since, by \eqref{eq412},
\[
(P\rho_{x}^{n,2},\rho_{x}^{n,3})
+ \bigl((1+H^{n})P_{0}r_{x}^{n,2},r_{x}^{n,3}\bigr) = \tfrac{1}{2}\gamma_{2}^{n,2},
\]
we have
\[
\beta_{7}^{n} = -\tfrac{1}{72}\bigl[\bigl((1+H^{n})P_{0}r_{x}^{n,2},P_{0}r_{x}^{n,3}-r_{x}^{n,3}\bigr)
 + \tfrac{1}{2}\gamma_{2}^{n,2}\bigr],
\]
and, therefore, in view of Lemma \ref{L21}(ii) and \eqref{eq413},
\begin{equation}
\abs{\beta_{7}^{n}} \leq \tfrac{C}{h^{6}}(\|\ve^{n}\|^{2} + \|e^{n}\|^{2}).
\label{eq454}
\end{equation}
Finally, $\beta_{8}^{n}$ is given by \[
\beta_{8}^{n} = \tfrac{1}{24^{2}}
\bigl[\|P\rho_{x}^{n,3}\|^{2} +\bigl((1+H^{n})P_{0}r_{x}^{n,3},P_{0}r_{x}^{n,3}\bigl)\bigr].
\]
Hence
\[
\beta_{8}^{n} \leq \tfrac{1}{24^{2}}(\|\rho_{x}^{n,3}\|^{2} + C'\|r_{x}^{n,3}\|^{2}),
\]
and from \eqref{eq43}, \eqref{eq44}, and the inverse properties of 
$S_{h}$, $S_{h,0}$, we get that
\begin{equation}
\beta_{8}^{n} \leq \tfrac{C_{0}}{h^{2}}(\|P\rho_{x}^{n,2}\|^{2} + \|P_{0}r_{x}^{n,2}\|^{2}).
\label{eq455}
\end{equation}
where $C_{0}$ is a constant independent of $h$ and $k$.  We conclude therefore from
\eqref{eq445}-\eqref{eq455} that
\begin{align*}
\|f^{n}\|^{2} + \bigl((1+H^{n})g^{n},g^{n}\bigr)
& \leq \|\ve^{n}\|^{2} + \bigl((1 + H^{n})e^{n},e^{n}\bigr)
 + C_{\lambda}k(\|\ve^{n}\|^{2} + \|e^{n}\|^{2}) \\
& \,\,\,\,\,\,
+ k^{6}\bigl(\lambda^{2}C_{0} - \tfrac{C_{\alpha}}{72}\bigr)(\|P\rho_{x}^{n,2}\|^{2}
+ \|P_{0}r_{x}^{n,2}\|^{2}).
\end{align*}
(v)\, {\em{Stability, use of local error estimates, and completion of the proof}} \\
From the last inequality above, for $\lambda \leq \lambda_{0}=\sqrt{C_{\alpha}/(72C_{0})}$ it follows
that
\begin{equation}
\|f^{n}\|^{2} + \bigl((1+H^{n})g^{n},g^{n}\bigr) \leq
\|\ve^{n}\|^{2} + \bigl((1 + H^{n})e^{n},e^{n}\bigr) +
C_{\lambda}k(\|\ve^{n}\|^{2} + \|e^{n}\|^{2}).
\label{eq456}
\end{equation}
Therefore, using the equations \eqref{eq442}, we see that
\begin{equation}
\begin{aligned}
\|\ve^{n+1}\|^{2} + \bigl((1 + H^{n+1})e^{n+1},e^{n+1}\bigr) & = \|f^{n}\|^{2}
+ 2(f^{n},f_{1}^{n} + \delta_{1}^{n}) + \|f_{1}^{n} + \delta_{1}^{n}\|^{2}
+ \bigl((1+H^{n+1})g^{n},g^{n}\bigr)\\
& + 2\bigl((1+H^{n+1})g^{n},g_{1}^{n}+\delta_{2}^{n}\bigr)
+ \bigl((1+H^{n+1})(g_{1}^{n} + \delta_{2}^{n}),g_{1}^{n} + \delta_{2}^{n}\bigr).
\end{aligned}
\label{eq457}
\end{equation}
From \eqref{eq443}, \eqref{eq444} for $0\leq n\leq n^{*}$, we obtain
\[
\|f^{n}\|\|f_{1}^{n}\| + \|g^{n}\|\|g_{1}^{n}\| \leq C_{\lambda}k(\|\ve^{n}\|^{2} + \|e^{n}\|^{2}),
\]
and from Proposition \ref{P31}, and \eqref{eq443}, \eqref{eq444} that
\[
\|f^{n}\|\|\delta_{1}^{n}\| + \|g^{n}\|\|\delta_{2}^{n}\| \leq
C_{\lambda}k(\|\ve^{n}\|^{2} + \|e^{n}\|^{2} + (h^{r-1} + k^{4})^{2}).
\]
Moreover, taking into account that
\[
\bigl((1+H^{n+1})g^{n},g^{n}\bigr) \leq \bigl((1+H^{n})g^{n},g^{n}\bigr) + Ck\|g^{n}\|^{2},
\]
we get from \eqref{eq457} in view of \eqref{eq456} and \eqref{eq210},
\[
\|\ve^{n+1}\|^{2} + \bigl((1 + H^{n+1})e^{n+1},e^{n+1}\bigr) \leq
(1 + \tfrac{C_{\lambda}}{C_{\alpha}}k)\Bigl(\|\ve^{n}\|^{2}
+ \bigl((1 + H^{n})e^{n},e^{n}\bigr)\Bigr) + C_{\lambda}'k(h^{r-1}+k^{4})^{2},
\]
for $0\leq n\leq n^{*}$. Therefore from Gronwall's lemma it follows that
\[
\|\ve^{n}\|^{2} + \bigl((1+H^{n})e^{n},e^{n}\bigr)
\leq C_{1}\Bigl((\|\ve^{0}\|^{2} + \bigl((1+H^{0})e^{0},e^{0}\bigr)\Bigr)
+ C_{2}(h^{r-1} + k^{4})^{2},
\]
where $C_{1}$, $C_{2}$ do not depend on $n^{*}$. Therefore, by \eqref{eq210}
\[
\|\ve^{n}\|^{2} + \|e^{n}\|^{2} \leq C_{1}(\|\ve^{0}\|^{2} + \|e^{0}\|^{2})
+ C_{2}(h^{r-1} + k^{4})^{2},
\]
i.e.
\[
\|\ve^{n}\| + \|e^{n}\| \leq C(h^{r-1} + k^{4}),
\]
for $0\leq n\leq n^{*} + 1$, where the constant $C$ does not depend on $n^{*}$. From the inverse
inequalities of $S_{h}$, $S_{h,0}$ and the fact that $r\geq 3$ it follows that $n^{*}$ was not maximal.
Hence, we may take $n^{*}=M-1$ and obtain the result of the theorem.
\end{proof}
\section{computational remarks} In this section we present results of numerical 
experiments that we performed in order to determine computationally the spatial and 
temporal rates of convergence of fully discrete schemes of the type analyzed in the previous sections. 
We also report on some computational results on the validity 
of the property \eqref{eq212} in the case of cubic and quartic splines. \\ \noindent
(i)\,\, \emph{Spatial rates of convergence}
\\ \noindent
As previously mentioned, it is well known that in the case of first-order hyperbolic
problems, the standard 
Galerkin method on a general quasiuniform mesh converges in $L^{2}$ with a spatial rate of $r-1$. 
We illustrate this for the problem at hand in the case of $C^{2}$ cubic splines $(r=4)$ defined on the quasiuniform mesh 
$0=x_{1}<x_{2}<...<x_{N+1}=1$, where $x_{i+1}=x_{i} + h_{i}$, $1\leq i\leq N$, $N$ even,
and $h_{i}=0.8h$ if $i\equiv 0\bmod 2$, $h_{i}=1.2h$ if $i\equiv 1\bmod 2$, and $h=1/N$. 
We solve the system of shallow water equations \eqref{eqsw} with the addition 
of a suitable right-hand side and initial conditions, so that its exact solution
is $\eta(x,t) = \exp(2t)(x+\cos(\pi x)+2)$, $u(x,t) = \exp(-xt)\sin(\pi x)$.
We integrate the semidiscrete problem in time for $0\leq t\leq 1$ by the classical 
RK4 scheme taking small enough time steps so that the temporal error is negligible in
comparison with the spatial one. Table \ref{tbl1} shows the numerical rates of convergence
at $t=1$ in the $L^{2}$ and $L^{\infty}$ norms and the $H^{1}$ seminorm as $N$ increases,
when $k/h=1/20$. The $L^{2}$ and $L^{\infty}$ rates are practically equal to $3$, while
the $H^{1}$ seminorm rate is practically $2$. (The analogous experiment with $C^{4}$ quintic
splines $(r=6)$ yielded numerical rates of convergence in $L^{2}$, $L^{\infty}$, and
$H^{1}$ approximately equal to $5$, $5$, and $4$ respectively.)
 \par
In the case of uniform spatial mesh the numerical experiments suggest that the 
$L^{2}$ rate of convergence is $O(h^{r})$, i.e. optimal. This was proved in \cite{ad}
for the finite element space of continuous piecewise linear functions $(r=2)$ for 
\eqref{eqsw} and for general $r$ in the case of periodic boundary conditions. 
Table 4.2 in \cite{ad} suggests that the numerical $L^{2}$ rates of convergence for $C^{2}$ cubic splines are also optimal, i.e. equal to 4. 
\\ \noindent
Here we illustrate this property in the
case of $C^{4}$ quintic splines. Table \ref{tbl2} shows  the associated numerical rates with
$h=1/N$, $k=10^{-4}$ for the same test problem at $t=1$. The $L^{2}$, $L^{\infty}$, 
$H^{1}$ rates are observed to be close to 6, 6, and 5, for both components.
\\ \noindent
(ii)\,\, \emph{Temporal rates of convergence} \par
We turn now to the computational determination of the temporal accuracy, which is a
harder exercise. We follow the technique proposed in \cite{bdkmc}. We select a test 
problem with known exact solution and for a fixed spatial grid (i.e fixed $h$) we
compute the numerical solutions up to $t=T$ with decreasing values of $k=T/M$ 
satisfying the stability condition. The $L^{2}$ error $E=E(T)$ ceases to decrease
of course after a certain $k$ when the temporal error becomes much smaller than the 
spatial one. Denote by $V^{M_{ref}}(h,k_{ref})$ the numerical solution (here $V=\eta$
or $u$) computed with a time step $k_{ref} = T/M_{ref}$ which is taken well below the
threshold after which $E$ stabilizes. Therefore, the error of the approximation
$V^{M_{ref}}(h,k_{ref})$ is almost purely spatial. We then compute a modified 
$L^{2}$ error for values of $k$ much larger than $k_{ref}$, which is defined by
\[
E^{*} = E^{*}(T) = \|V^{M}(h,k) - V^{M_{ref}}(h, k_{ref})\|,
\]
where $T = Mk$. It is reasonable to expect that the subtraction 
$V^{M}(h,k) - V^{M_{ref}}(h,k_{ref})$ will essentially cancel the spatial error of
$V^{M}(h,k)$ for a range of values of $k$; thus the temporal order of accuracy of the 
scheme may emerge from a sequence of computations of $E^{*}$ with decreasing $k$ in 
that range. The success of this procedure depends of course on finding an appropriate 
range of time steps depending on the chosen spatial grid, the solution of the test
problem, $k_{ref}$, and the order of magnitude of the errors. For scalar problems and
time-stepping schemes with weak stability conditions, such as those considered in \cite{bdkmc},
this technique works rather well. In the case of systems of pde's and a high-order
conditionally stable scheme, such as the one at hand, one has to experiment
considerably; among other we found that the test problems should be chosen so that
the errors of all components of the system (here $\eta$ and $u$) are of the same order
of magnitude. The results of our experiments are shown in Table \ref{tbl3} for cubic 
and quintic splines on uniform and quasiuniform spatial meshes. The exact solution was 
taken now to be $\eta = \exp(-4t^{2})(x + \cos(\pi x))$, $u = \exp(-tx)\sin(\pi x)$, and
corresponding right-hand sides and initial conditions were found. The errors and temporal
rates at $T=1$ were computed with uniform mesh with $h=1/N$ and the quasiuniform mesh
defined in part (i) of this section. For each $M=T/k$ we show the modified $L^{2}$
error $E^{*}$ and the corresponding numerical temporal rate of convergence. In all cases
we took $M_{ref}=600$; the $L^{2}$ error of $V^{M_{ref}}(h,k_{ref})$ is denoted by 
$E_{ref}$. The fourth-order temporal convergence emerges in all cases. (The experiments
also gave fourth-order temporal convergence when RK4 was coupled with continuous,
piecewise linear spatial discretizations. In all cases the spatial grid was taken coarse 
enough so that the spatial errors not be too small.) \\ \noindent
(iii)\,\, \emph{Remarks on the validity of \eqref{eq212}} \par
In closing, we report on a few numerical experiments we performed in order to check the validity of the hypothesis \eqref{eq212} in the case of $C^{2}$ cubic $(r=4)$ and 
$C^{3}$ quartic $(r=5)$ splines. To this effect we computed the $H^{3}(0,1)$ error
$\|Pv - v\|_{3}$ for a $C^{\infty}$ function $v$ and a function that was $C^{2}$ 
and piecewise $C^{3}$, i.e. so that $v\in H^{3}$ but $v \notin H^{4}$, and found its
numerical rate of convergence as $h\to 0$ in the case of uniform and quasiuniform meshes.
In all cases we found that $\|Pv - v\|_{3}$ was of $O(h^{\alpha})$ with $\alpha >0$,
which suggests that $\|Pv\|_{3}\leq C(v)$ for a function $v$ that is at least in $H^{3}$.
In the case of a smooth $v$ (we took $v(x) = \sin (\pi x /2 + 1)$ the numerical rate of
convergence of $\|Pv - v\|_{3}$ was found to be optimal, i.e. equal to one for cubic
and to two for quartic splines, for uniform and quasiuniform meshes. (Results not shown.)
\par
We then experimented with the $C^{2}$ function whose third derivative is given by
\[
v'''(x) = 
\begin{cases}
\exp(x) \,\, ,   \,\, & \,\,\,\,\,\, 0\leq x < 1/4 \\
\sin(\pi x)\,\,\, , \,\, &   1/4 \leq x < 1/2, \\
\exp(-x) \,\, ,  \,\,  & 1/2 \leq x < 3/4,\\
\cos(\pi x) \,\, , \,\,&   3/4 \leq x \leq 1.
\end{cases}
\]
The grids that we considered were uniform with $h=1/N$ and quasiuniform with 
$x_{i+1} = x_{i} + h_{i}$, $1\leq i\leq N$, 
$h_{i}=3h/2$, if $i\equiv 0\mod 2$, and $h_{i}=h/2$, if $i\equiv 1\mod 2$, $N$ odd 
and $h = 2/(2N - 1)$. (We mainly took $N$ odd so that the discontinuities
of $v$ did not occur at meshpoints. For $N$ even we took $h=1/N$.)
\par
In Table \ref{tbl4} we show the results obtained in the case of cubic splines on a uniform grid
with odd $N$. The order of convergence $\alpha$ was found to be approximately equal to
$0.5$. (The table also shows the errors and rates of convergence for a variety of other
norms and seminorms.) The same rates of convergence were found (results not shown) in
the case of the quasiuniform grid with odd $N$. (In the case of even $N$ optimal-order 
results were found, i.e. $\alpha=1$, for both uniform and quasiuniform meshes.)
\par
In the case of quartic splines in all cases of uniform and quasiuniform meshes with odd
or even $N$ we observed, as expected, that $\alpha$ was approximately equal to 0.5, due
to the restricted regularity of $v$. We just show in Table \ref{tbl5} the results on the
quasiuniform grid with $N$ odd.
\normalsize \vspace{-3pt}
\subsection*{Acknowledgment} This work was partially supported by the project 
``Innovative Actions in Environmental Research and Development (PErAn)'' (MIS 5002358),
implemented under the ``Action for the strategic development of the Research and 
Technological sector'' funded by the Operational Program ``Competitiveness, 
and Innovation'' (NSRF 2014-2020) and cofinanced by Greece and the EU (European Regional 
Development Fund).
\normalsize
\bibliographystyle{amsplain} %    Insert the bibliography data here.

\begin{table}[b]
\begin{subtable}{1\textwidth}
\centering
\begin{tabular}{ | c | c  c | c  c | c  c | }\hline 
\rule{0pt}{10pt} $N$ & $L^{2}$ error &   rate & $L^{\infty}$ error & rate & $H^{1}$ seminorm  error &  rate\\ \hline 
\rule{0pt}{10pt} 160 & 1.1057e-06 &  -    & 2.4537e-06 &   -   & 5.8016e-04 &   -   \\ \hline
\rule{0pt}{10pt} 200 & 5.6700e-07 & 2.993 & 1.2600e-06 & 2.987 & 3.6898e-04 & 2.028 \\ \hline
\rule{0pt}{10pt} 240 & 3.2848e-07 & 2.994 & 7.2837e-07 & 3.006 & 2.5514e-04 & 2.024 \\ \hline
\rule{0pt}{10pt} 280 & 2.0700e-07 & 2.996 & 4.5857e-07 & 3.002 & 1.8686e-04 & 2.020 \\ \hline
\rule{0pt}{10pt} 320 & 1.3875e-07 & 2.996 & 3.0686e-07 & 3.008 & 1.4273e-04 & 2.017 \\ \hline
\rule{0pt}{10pt} 360 & 9.7479e-08 & 2.998 & 2.1566e-07 & 2.995 & 1.1256e-04 & 2.017 \\ \hline 
\rule{0pt}{10pt} 400 & 7.1102e-08 & 2.995 & 1.5726e-07 & 2.998 & 9.1058e-05 & 2.012 \\ \hline
\rule{0pt}{10pt} 440 & 5.3431e-08 & 2.998 & 1.1834e-07 & 2.983 & 7.5161e-05 & 2.013 \\ \hline
\end{tabular}
\caption*{(a)}\label{tb1a}
\end{subtable}
\begin{subtable}{1\textwidth}
\centering
\begin{tabular}{ | c | c  c | c  c | c  c | }\hline 
\rule{0pt}{10pt} $N$ & $L^{2}$ error &   rate & $L^{\infty}$ error & rate & $H^{1}$ seminorm  error &  rate\\ \hline 
\rule{0pt}{10pt} 160 & 2.3101e-08 &  -    & 4.9500e-08 &   -   & 1.1641e-05 &   -   \\ \hline
\rule{0pt}{10pt} 200 & 1.1881e-08 & 2.980 & 2.4909e-08 & 3.078 & 7.4840e-06 & 1.980 \\ \hline
\rule{0pt}{10pt} 240 & 6.8975e-09 & 2.983 & 1.4296e-08 & 3.046 & 5.2139e-06 & 1.982 \\ \hline
\rule{0pt}{10pt} 280 & 4.3513e-09 & 2.989 & 8.8425e-09 & 3.116 & 3.8374e-06 & 1.989 \\ \hline
\rule{0pt}{10pt} 320 & 2.9189e-09 & 2.990 & 5.8042e-09 & 3.153 & 2.9420e-06 & 1.990 \\ \hline
\rule{0pt}{10pt} 360 & 2.0516e-09 & 2.994 & 4.0743e-09 & 3.005 & 2.3263e-06 & 1.994 \\ \hline 
\rule{0pt}{10pt} 400 & 1.4972e-09 & 2.990 & 2.9627e-09 & 3.024 & 1.8863e-06 & 1.990 \\ \hline
\rule{0pt}{10pt} 440 & 1.1255e-09 & 2.994 & 2.2195e-09 & 3.030 & 1.5597e-06 & 1.994 \\ \hline
\end{tabular}
\caption*{(b)}\label{tb1b}
\end{subtable}
\caption{Spatial rates of convergence, cubic splines, quasiuniform mesh,
$T=1$, $\frac{k}{h}=\frac{1}{20}$, (a): $\eta$, (b): $u$} \label{tbl1}
\end{table}
\begin{table}[b]
\begin{subtable}{1\textwidth}
\centering
\begin{tabular}{ | c | c  c | c  c | c  c | }\hline 
\rule{0pt}{10pt} $N$ & $L^{2}$ error &   rate & $L^{\infty}$ error & rate & $H^{1}$ seminorm  error &  rate\\ \hline 
\rule{0pt}{10pt} 12 & 5.5379e-07 &  -    & 1.4510e-06 &   -   & 4.2901e-05 &   -   \\ \hline
\rule{0pt}{10pt} 18 & 4.7013e-08 & 6.083 & 1.2278e-07 & 6.091 & 4.7221e-06 & 5.442 \\ \hline
\rule{0pt}{10pt} 24 & 8.2765e-09 & 6.038 & 2.1654e-08 & 6.032 & 1.0096e-06 & 5.362 \\ \hline
\rule{0pt}{10pt} 30 & 2.1511e-09 & 6.038 & 5.6238e-09 & 6.042 & 3.0752e-07 & 5.327 \\ \hline
\rule{0pt}{10pt} 36 & 7.1581e-10 & 6.035 & 1.8738e-09 & 6.028 & 1.1680e-07 & 5.310 \\ \hline
\end{tabular}
\caption*{(a)}\label{tb2a}
\end{subtable}
\begin{subtable}{1\textwidth}
\centering
\begin{tabular}{ | c | c  c | c  c | c  c | }\hline 
\rule{0pt}{10pt} $N$ & $L^{2}$ error &   rate & $L^{\infty}$ error & rate & $H^{1}$ seminorm  error &  rate\\ \hline 
\rule{0pt}{10pt} 12 & 9.2535e-09 &  -    & 2.1916e-08 &   -   & 4.4551e-07 &   -   \\ \hline
\rule{0pt}{10pt} 18 & 7.8813e-10 & 6.075 & 1.8705e-09 & 6.070 & 5.7648e-08 & 5.043 \\ \hline
\rule{0pt}{10pt} 24 & 1.4005e-10 & 6.005 & 3.3366e-10 & 5.992 & 1.3670e-08 & 5.003 \\ \hline
\rule{0pt}{10pt} 30 & 3.6090e-11 & 6.077 & 8.7567e-11 & 5.995 & 4.4472e-09 & 5.032 \\ \hline
\rule{0pt}{10pt} 36 & 1.1975e-11 & 6.051 & 2.9254e-11 & 6.014 & 1.7807e-09 & 5.020 \\ \hline
\end{tabular}
\caption*{(b)}\label{tb2b}
\end{subtable}
\caption{Spatial rates of convergence, quintic splines $T=1$, uniform mesh,
$h = 1/N$, $k=10^{-4}$, (a): $\eta$, (b): $u$} \label{tbl2}
\end{table} 
\begin{table}[b]
\begin{subtable}{1\textwidth}
\centering
\begin{tabular}{ c | c  c | c c |}\cline{2-5}
& \multicolumn{2}{c|}{$\eta$}  & \multicolumn{2}{c |}{$u$} \\ \hline
\multicolumn{1}{|c|}{M}   &  $E^{*}$   \,& rate  & $E^{*}$    \,& rate \\ \hline
\multicolumn{1}{|c|}{110} & 2.5095e-08 \,&   -   & 2.3825e-08 \,&   -   \\ \hline
\multicolumn{1}{|c|}{115} & 2.1068e-08 \,& 3.934 & 1.9943e-08 \,& 4.001 \\ \hline
\multicolumn{1}{|c|}{120} & 1.7814e-08 \,& 3.942 & 1.6825e-08 \,& 3.994 \\ \hline
\multicolumn{1}{|c|}{125} & 1.5163e-08 \,& 3.947 & 1.4296e-08 \,& 3.990 \\ \hline
\multicolumn{1}{|c|}{130} & 1.2987e-08 \,& 3.950 & 1.2225e-08 \,& 3.990 \\ \hline
\multicolumn{1}{|c|}{135} & 1.1188e-08 \,& 3.950 & 1.0515e-08 \,& 3.992 \\ \hline
\multicolumn{1}{|c|}{140} & 9.6915e-09 \,& 3.949 & 9.0931e-09 \,& 3.996 \\ \hline
\multicolumn{1}{|c|}{145} & 8.4378e-09 \,& 3.948 & 7.9024e-09 \,& 4.000 \\ \hline
\multicolumn{1}{|c|}{150} & 7.3808e-09 \,& 3.948 & 6.8997e-09 \,& 4.002 \\ \hline
\multicolumn{1}{|c}{$E_{ref}$} & \multicolumn{4}{c|}{}\\ \hline
\multicolumn{1}{|c|}{600} & 7.6301e-09 \,&   -   & 4.9031e-09 \,&   -   \\ \hline
\end{tabular}
\caption*{(a) Cubic splines, uniform mesh, $N=60$}
\end{subtable}
\begin{subtable}{1\textwidth}
\centering
\begin{tabular}{ c | c  c | c c |}\cline{2-5}
& \multicolumn{2}{c|}{$\eta$}  & \multicolumn{2}{c |}{$u$} \\ \hline
\multicolumn{1}{|c|}{M}   &  $E^{*}$   \,& rate  & $E^{*}$    \,& rate \\ \hline
\multicolumn{1}{|c|}{105} & 3.0062e-08 \,&   -   & 2.8915e-08 \,&   -   \\ \hline
\multicolumn{1}{|c|}{110} & 2.4975e-08 \,& 3.985 & 2.4051e-08 \,& 3.959 \\ \hline
\multicolumn{1}{|c|}{115} & 2.0905e-08 \,& 4.002 & 2.0182e-08 \,& 3.946 \\ \hline
\multicolumn{1}{|c|}{120} & 1.7619e-08 \,& 4.018 & 1.7073e-08 \,& 3.930 \\ \hline
\multicolumn{1}{|c}{$E_{ref}$} & \multicolumn{4}{c|}{}\\ \hline
\multicolumn{1}{|c|}{600} & 2.2953e-06 \,&   -   & 8.2945e-07 \,&   -   \\ \hline
\end{tabular}
\caption*{(b) Cubic splines, quasiuniform mesh, $N=60$}
\end{subtable}
\begin{subtable}{1\textwidth}
\centering
\begin{tabular}{ c | c  c | c c |}\cline{2-5}
& \multicolumn{2}{c|}{$\eta$}  & \multicolumn{2}{c |}{$u$} \\ \hline
\multicolumn{1}{|c|}{M}  &  $E^{*}$   \,& rate  & $E^{*}$    \,& rate \\ \hline
\multicolumn{1}{|c|}{60} & 2.7218e-07 \,&   -   & 2.6114e-07 \,&   -   \\ \hline
\multicolumn{1}{|c|}{65} & 1.9786e-07 \,& 3.984 & 1.9000e-07 \,& 3.974 \\ \hline
\multicolumn{1}{|c|}{70} & 1.4716e-07 \,& 3.995 & 1.4164e-07 \,& 3.963 \\ \hline
\multicolumn{1}{|c|}{75} & 1.1167e-07 \,& 3.999 & 1.0776e-07 \,& 3.963 \\ \hline
\multicolumn{1}{|c|}{80} & 8.6261e-08 \,& 4.001 & 8.3416e-08 \,& 3.967 \\ \hline
\multicolumn{1}{|c|}{85} & 6.7679e-08 \,& 4.002 & 6.5559e-08 \,& 3.973 \\ \hline
\multicolumn{1}{|c|}{95} & 4.3353e-08 \,& 4.004 & 4.2123e-08 \,& 3.977 \\ \hline
\multicolumn{1}{|c|}{100} & 3.5312e-08 \,& 4.000 & 3.4364e-08 \,& 3.969 \\ \hline
\multicolumn{1}{|c}{$E_{ref}$} & \multicolumn{4}{c|}{}\\ \hline
\multicolumn{1}{|c|}{600} & 2.2956e-09 \,&   -   & 6.7454e-10 \,&   -   \\ \hline
\end{tabular}
\caption*{(c) Quintic splines, uniform mesh, $N=20$}
\end{subtable}
\begin{subtable}{1\textwidth}
\centering
\begin{tabular}{ c | c  c | c c |}\cline{2-5}
& \multicolumn{2}{c|}{$\eta$}  & \multicolumn{2}{c |}{$u$} \\ \hline
\multicolumn{1}{|c|}{M}   &  $E^{*}$   \,& rate  & $E^{*}$    \,& rate \\ \hline
\multicolumn{1}{|c|}{80}  & 8.5189e-08 \,&   -   & 8.6138e-08 \,&   -   \\ \hline
\multicolumn{1}{|c|}{85}  & 6.6830e-08 \,& 4.004 & 6.7616e-08 \,& 3.994 \\ \hline
\multicolumn{1}{|c|}{90}  & 5.3130e-08 \,& 4.013 & 5.3949e-08 \,& 3.950 \\ \hline
\multicolumn{1}{|c|}{95}  & 4.2745e-08 \,& 4.023 & 4.3591e-08 \,& 3.943 \\ \hline
\multicolumn{1}{|c|}{100} & 3.4762e-08 \,& 4.031 & 3.5620e-08 \,& 3.937 \\ \hline
\multicolumn{1}{|c|}{105} & 2.8548e-08 \,& 4.036 & 2.9397e-08 \,& 3.935 \\ \hline
\multicolumn{1}{|c}{$E_{ref}$} & \multicolumn{4}{c|}{}\\ \hline
\multicolumn{1}{|c|}{600} & 1.3611e-08 \,&   -   & 4.6738e-09 \,&   -   \\ \hline
\end{tabular}
\caption*{(d) Quintic splines, quasiuniform grid, $N=30$}
\end{subtable}
\caption{Temporal rates of convergence, $T=1$} 
\label{tbl3}
\end{table}
\scriptsize
\begin{table}[b]
\centering
\setlength\tabcolsep{4pt}
\begin{tabular}{ | c | l c | l c | l c | l c | l c | l c |}\hline
\rule{0pt}{6.2pt} $N$ & $\|Pv - v\|$ & order & $\abs{Pv-v}_{1}$ & order & $\abs{Pv-v}_{2}$ & order & $\abs{Pv-v}_{3}$ & order & $\|Pv-v\|_{3}$ & order & $\|Pv - v\|_{\infty}$ & order \\ \hline
9 & 3.38e-06 & - & 7.35e-06 & - & 1.30e-04 & - & 5.45e-03 & - & 2.66e-01 & - & 3.63e-01 & - \\ \hline
17 & 3.56e-07 & 3.542 & 2.25e-05 & 2.751 & 1.67e-03 & 1.856 & 1.76e-01 & 0.649 & 2.17e-01 & 0.811 & 1.10e-06 & 2.984\\ \hline
33    & 3.39e-08 & 3.546 & 4.20e-06 & 2.532 & 6.09e-04 & 1.525 & 1.25e-01 & 0.513 & 1.45e-01 & 0.607 & 1.48e-07 & 3.030\\ \hline
65    & 3.16e-09 & 3.498 & 7.76e-07 & 2.492 & 2.21e-04 & 1.493 & 8.91e-02 & 0.500 & 9.88e-02 & 0.564 & 1.93e-08 & 3.000\\ \hline
129   & 2.88e-10 & 3.496 & 1.40e-07 & 2.496 & 7.93e-05 & 1.498 & 6.32e-02 & 0.501 & 6.80e-02 & 0.545 & 2.47e-09 & 2.998\\ \hline
257   & 2.58e-11 & 3.498 & 2.51e-08 & 2.498 & 2.82e-05 & 1.499 & 4.48e-02 & 0.500 & 4.71e-02 & 0.532 & 3.13e-10 & 2.999\\ \hline
513   & 2.30e-12 & 3.499 & 4.45e-09 & 2.499 & 1.00e-05 & 1.499 & 3.17e-02 & 0.500 & 3.29e-02 & 0.522 & 3.94e-11 & 2.999\\ \hline
1025  & 2.04e-13 & 3.499 & 7.90e-10 & 2.499 & 3.54e-06 & 1.500 & 2.24e-02 & 0.500 & 2.30e-02 & 0.515 & 4.94e-12 & 3.000\\ \hline
2049  & 1.81e-14 & 3.500 & 1.40e-10 & 2.500 & 1.25e-06 & 1.500 & 1.59e-02 & 0.500 & 1.61e-02 & 0.511 & 6.18e-13 & 3.000\\ \hline
4097  & 1.60e-15 & 3.499 & 2.47e-11 & 2.500 & 4.44e-07 & 1.500 & 1.12e-02 & 0.500 & 1.14e-02 & 0.508 & 7.73e-14 & 3.001\\ \hline
\end{tabular}
\caption{Errors $Pv-v$ and order of convergence, non smooth $v$, 
cubic splines, uniform mesh}\label{tbl4}
\end{table}
\normalsize
\scriptsize
\begin{table}[b]
\centering
\setlength\tabcolsep{4pt}
\begin{tabular}{ | c | l c | l c | l c | l c | l c | l c |}\hline
\rule{0pt}{6.2pt} $N$ & $\|Pv - v\|$ & order & $\abs{Pv-v}_{1}$ & order & $\abs{Pv-v}_{2}$ & order & $\abs{Pv-v}_{3}$ & order & $\|Pv-v\|_{3}$ & order & $\|Pv - v\|_{\infty}$ & order \\ \hline
9 & 1.32e-06 & - & 3.24e-06 & - & 7.11e-05 & - & 5.86e-03 & - & 3.39e-01 & - & 4.46e-01 & - \\ \hline
17    & 1.71e-07 & 3.216 & 1.17e-05 & 2.841 & 1.19e-03 & 2.505 & 1.31e-01 & 1.497 & 1.60e-01 & 1.614 & 5.18e-07 & 2.881\\ \hline
33    & 1.63e-08 & 3.541 & 1.96e-06 & 2.692 & 2.68e-04 & 2.250 & 5.84e-02 & 1.216 & 6.73e-02 & 1.304 & 6.95e-08 & 3.029\\ \hline
65    & 1.46e-09 & 3.554 & 3.48e-07 & 2.546 & 9.31e-05 & 1.559 & 4.05e-02 & 0.539 & 4.47e-02 & 0.603 & 8.81e-09 & 3.046\\ \hline
129   & 1.31e-10 & 3.526 & 6.20e-08 & 2.519 & 3.30e-05 & 1.513 & 2.86e-02 & 0.505 & 3.07e-02 & 0.548 & 1.11e-09 & 3.022\\ \hline
257   & 1.16e-11 & 3.512 & 1.10e-08 & 2.509 & 1.17e-05 & 1.506 & 2.03e-02 & 0.502 & 2.13e-02 & 0.532 & 1.39e-10 & 3.010\\ \hline
513   & 1.03e-12 & 3.506 & 1.95e-09 & 2.504 & 4.14e-06 & 1.503 & 1.43e-02 & 0.501 & 1.48e-02 & 0.522 & 1.75e-11 & 3.005\\ \hline
1025  & 9.10e-14 & 3.503 & 3.45e-10 & 2.502 & 1.46e-06 & 1.501 & 1.01e-02 & 0.501 & 1.04e-02 & 0.515 & 2.19e-12 & 3.003\\ \hline
2049  & 8.05e-15 & 3.501 & 6.09e-11 & 2.501 & 5.17e-07 & 1.501 & 7.17e-03 & 0.500 & 7.29e-03 & 0.511 & 2.74e-13 & 3.001\\ \hline
4097  & 7.19e-16 & 3.486 & 1.08e-11 & 2.499 & 1.83e-07 & 1.500 & 5.07e-03 & 0.500 & 5.13e-03 & 0.507 & 3.42e-14 & 2.999\\ \hline
\end{tabular}
\caption{Errors $Pv-v$ and orders of convergence, non smooth $v$, 
quartic splines, quasi\-uniform mesh}\label{tbl5}
\end{table}
\end{document}